\documentclass[11pt,reqno]{amsart}
\usepackage{tikz}
\usepackage{subfig}
\usepackage{epstopdf}
\usepackage{epsfig}
\usepackage{rotating}
\graphicspath{{../figures/},{"../../Code/hierarchical models/FiguresAdamNew/"},{"../../Code/hierarchical models/Type0/"}}
\usepackage{tikz}
\usetikzlibrary{shapes,arrows}
\tikzstyle{decision} = [diamond, draw, fill=blue!20, 
    text width=4.5em, text badly centered, node distance=3cm, inner sep=0pt]
\tikzstyle{block} = [rectangle, draw, fill=blue!20, 
    text width=5em, text centered, rounded corners, minimum height=4em]
\tikzstyle{line} = [draw, -latex']
\tikzstyle{cloud} = [draw, ellipse,fill=red!20, node distance=3cm,
    minimum height=2em]
    
\usetikzlibrary{positioning}
\tikzset{main node/.style={circle,fill=blue!20,draw,minimum size=1cm,inner sep=0pt},  }

\usepackage{amsmath}
\numberwithin{equation}{section}
\usepackage{cite}
\usepackage{amsthm}
\usepackage{amssymb}
\usepackage{bbm}
\usepackage{mathrsfs}

\usepackage{mathtools}
\usepackage{BOONDOX-cal}
\usepackage{pst-node}
\usepackage{tikz-cd} 

\usepackage{soul}
\usepackage{color}
\usepackage{comment}
\usepackage{float}
\usepackage{IEEEtrantools}

\usepackage{hyperref}
\usepackage{cleveref}
\usepackage{graphicx} \usepackage{placeins}
\usepackage{essay-def}
\usepackage[titletoc]{appendix}

\title{Full extremal process in four-dimensional membrane model}
\author{Hao Ge}
\address{(Hao Ge) \href{mailto:haoge@pku.edu.cn}{haoge@pku.edu.cn}, Peking University}
\author{Xinyi Li}
\address{(Xinyi Li) \href{mailto:xinyili@bicmr.pku.edu.cn}{xinyili@bicmr.pku.edu.cn}, Peking University}
\author{Jiaxi Zhao}
\address{(Jiaxi Zhao) \href{mailto:jiaxi.zhao@u.nus.edu}{jiaxi.zhao@u.nus.edu}, National University of Singapore}

\begin{document}

\begin{abstract}
	We investigate the extremal process of four-dimensional membrane models as the size of the lattice $N$ tends to infinity. 
	We prove the cluster-like geometry of the extreme points and the existence as well as the uniqueness of the extremal process. The extremal process is characterized by a distributional invariance property and a Poisson structure with explicit formulas. Our approach follows the same philosophy as the two-dimensional Gaussian free field (2D GFF) via the comparison with the modified branching random walk. 
	The proofs leverage the ``Dysonization'' technique and a careful treatment of the correlation structure, which is more intricate than the 2D GFF case. As a by-product, we also obtain a simpler proof of a crucial sprinkling lemma.
\end{abstract}


\maketitle

\section{introduction}
\subsection{The Gaussian free field and membrane model}

In the realms of physics and probability theory, the Gaussian free field (GFF, gradient model) and membrane models (Laplacian model) stand out as extensively studied lattice approximations to continuous entities. Recent years have witnessed the establishment of numerous connections between the GFF and various other statistical physics models, including Schramm-Loewner evolution \cite{schramm2009contour}, conformal loop ensemble \cite{sheffield2012conformal}, and Gaussian multiplicative chaos \cite{kahane1985chaos}. Substantial theoretical strides have been made in understanding the GFF\cite{sheffield2007gaussian, schramm2009contour, Rodriguez2012PhaseTA}. These notable advancements include the limit law of the maximum, characteristics of level sets, and a comprehensive Poisson characterization of the full extremal process of GFF in 2D \cite{biskup2019intermediate, biskup2016extreme, biskup2014conformal, biskup2018full}, which is also intimately related to Liouville quantum gravity. 

Less is known about the membrane models, which can be viewed as a distribution over the space of functions on 4D lattice $h: \mbZ^4 \rightarrow \mbR$. A discrete 4D membrane models $\lbb h_v^{V} \rbb_{v \in \mbZ^4}$ on a domain $V \subset \mbZ^4$ with Dirichlet boundary condition $h_v^{V}\Big|_{V^c} = 0$ , is a mean zero Gaussian process indexed by $\mbZ^d$ which 
takes the value $0$ on $V^c$ and has a probability distribution function (PDF) as \cite{schweiger2020maximum}
\bequ
	\mbP\lp dh_{\cdot}^{V} \rp = \frac{1}{Z_{V}}\exp\lp -\half  \sum_{v \in \mbZ^4} \lp \Delta h_v^{V} \rp^2\rp \prod_{v \in V} dh_v^{V} \prod_{v \in V^c} \delta_0\lp dh_v^{V} \rp.
\eequ 
Here, we focus our attention on the 4D membrane models on a 4-cube $V_N := \lp \lb 0, N \rb \cap \mbZ \rp^4$ of side length $N$ and denote it by $h_v^{N}$. Here the discrete Laplacian is defined as
\bequn
	\Delta h_v^{V} = \sum_{u \sim v} h_u^{V} - 4h_v^{V},
\eequn
where $u \sim v$ denotes the vertices that are adjacent to $v$.

Note that the Hamiltonian of membrane models differs from the GFF by replacing the inner product between gradient with Laplacian. Throughout the paper, we will use $\lbb h_v^{\nabla, V} \rbb_{v \in \mbZ^2}$ to denote the 2D GFF with distribution
\bequ
	\mbP\lp dh_{\cdot}^{\nabla, V} \rp = \frac{1}{Z_{V}^{\nabla}}\exp\lp -\half  \sum_{v \in \mbZ^2} \lp \nabla h_v^{\nabla, V} \rp^2\rp \prod_{v \in V} dh_v^{V} \prod_{v \in V^c} \delta_0\lp dh_v^{V} \rp.
\eequ
Here the discrete gradient is defined as the following 
\bequn
	\nabla h_v^{\nabla, V} = \lp h_u^{\nabla, V} - h_v^{\nabla, V} \rp_{{u \sim v}}^T,
\eequn
where $u \sim v$ denotes that vertex $u$ is adjacent to $v$.

The work \cite{kurt2009maximum} establishes the entropy repulsion in the 4D 
membrane models and the expectation of the maximum of the field in square lattices $V_N := [0,N]^4 \cap \mbZ^4$, i.e.
\begin{equation}
	m_N := \max_{x\in V_N} h_v \sim 2\sqrt{2\gamma}\log N,
\end{equation}
with $\gamma = 8/\pi^2$ satisfies that the asymptotic of the Green
function with respect to the 4D membrane model satisfies $G^{V_N}(x, x) \sim \gamma \log N + O(1)$
as $N \rightarrow \infty$ for $x$ “deep” inside $V_N$.

The tightness of maximum sequence with respect to $N$ and the limiting law is proved in \cite{schweiger2020maximum} by leveraging a general criterion of 
log-correlated field settled in \cite{ding2017convergence}. More recently,  \cite{li2022intermediate} established the scaling limit of the intermediate level sets of the membrane models using second-moment 
calculation, parallel to  similar results on GFF\cite{biskup2019intermediate}.

With the rich structure of the extremal process of GFF in mind, a natural question arises: \textit{Does a similar full extremal process exist in the context of the 4D membrane model?} In this work, we provide 
a positive answer to the above question by establishing the convergence of the extremal process of the 4D membrane models. 

%\JX{I write a detailed summary on the established results on 4D membrane model in \Cref{sec:literature}. Should we move it to here?}

\subsection{Main results}\label{sec:Mainresults}
We consider the associated extremal process depending on $N, r$ with support on $[0, 1]^4 \times \mbR$ of 4D membrane models
\bequ\label{equ:def-extremal-process}
	\eta_{N, r}:= \sum_{x \in V_N} \mathbf{1}_{h_x = \max_{z \in \Lambda_r(x)} h_z} \delta_{x/N} \otimes \delta_{h_x- m_N},
\eequ
where $\Lambda_r(x):=\{z\in\mbZ^2: \lv z - x \rv_1 \leq r\}$. %Here the first indicator function refers to the site $x$ which is a local $2r$ maximum and the following 
%two Dirac symbols refer to the point mass support on $x/N \in \lb 0, 1 \rb^4, h_x- m_N \in \mbR.$ 
This defines a random measure on $[0, 1]^4 \times \mbR$ given explicitly by
\bequ
	\eta_{N, r}(A \times B) = \sum_{x \in V_N} \mathbf{1}_{h_x = \max_{z \in \Lambda_r(x)} h_z} \mathbf{1}_{x/N \in A} \mathbf{1}_{h_x- m_N \in B}, \quad A \subset [0, 1]^4, B \subset \mbR.
\eequ
Given a function $f \in C_c\lp \lb 0, 1 \rb^4 \times \mbR \rp$, we define the inner product between $\eta_{N, r}$ and $f$ as 
\bequn
	\la \eta_{N, r}, f \ra := \sum_{x \in V_N} f(x/N, h_x-m_N).
\eequn
The main result of this paper is the following characterization of the extremal process as a Poisson point process.
\begin{Thm}[Poisson structure of limit extremal process]\label{thm:Poisson-limit}
There is a unique random Borel measure $Z (dx)$ on $[0, 1]^4$ satisfying $Z([0, 1]^4) < \infty$ a.s.\ and $Z(A) > 0$ 
a.s.\ for any open set $A \subset [0,1]^4$ such that $\forall r_N \rightarrow \infty$, any weak subsequential limit $\eta$ 
of the processes $\eta_{N, r_N}, N = 1, 2, \cdots$ takes the form of a PPP, i.e.
\bequ\label{equ:poisson-limit}
	\eta_{N, r_N} \rightarrow \eta\overset{\text{law}}{=} \PPP\lp Z\lp dx \rp \otimes e^{-\pi t} dt\rp.
\eequ
\end{Thm}

The key ingredients of the proof are precise analysis of the geometry of the maxima of 4D membrane models, akin to the approach used for the Gaussian Free Field (GFF)\cite{ding2014extreme}. We employ the critical covariance structure and the technique of distributional invariance following the methodology of \cite{biskup2016extreme}. The geometry of the near maxima landscape of the 4D membrane model is stated below.
\begin{Thm}[Geometry of the set of near maxima for 4D membrane model]\label{thm:4Dm-geo}
		There exists an absolute constant $c > 0$ such that
		\bequ
			\begin{aligned}
				\lim_{r \rightarrow \infty}\lim_{N \rightarrow \infty} \mbP\lp  \exists v, u \in V_N: r \leq \lv v - u \rv \leq \frac{N}{r}, \   h_u^{N}, h_v^{N} \geq m_N- c\log\log r  \rp = 0.
			\end{aligned}
		\eequ
\end{Thm}

This theorem arises from an in-depth analysis of the covariance structure of the 4D membrane model (see \Cref{prop:4D-cov}). The covariance matrix of this model is represented by the discrete Green's function associated with the bi-Laplacian operator. Unlike the Green's function for the Gaussian Free Field (GFF), which has explicit interpretations in discrete harmonic analysis and potential theory \cite{lawler2013intersections}, the Green's function for the 4D membrane model lacks a probabilistic interpretation. This discrepancy presents a significant challenge in characterizing the extremal process limit for 4D membrane models.

We addressed this challenge through a detailed analysis of the field correlation structure, as demonstrated in \Cref{lem:sprinkling}, referred to as the sprinkling lemma. While the sprinkling lemma for GFF was initially established in \cite{ding2014extreme}, we offer a simpler proof in this paper. 

The geometry interpretation of \Cref{thm:4Dm-geo} can be summarized as the sites of a membrane that attains the high value that is only smaller than $m_N$ by a constant are cluster-like, i.e. 
two points in this set are either close to each other ($r > \lv v - u \rv$) or far from each other 
($\lv v - u \rv > \frac{N}{r}$). Based on this theorem, one can derive an important property of the limiting measure of the extremal process, i.e., invariance under Dysonization. 
\begin{Thm}[Invariance under Dysonization]\label{thm:dysoninv}
	For any subsequential limit $\eta$ and any function $f \in C_c\lp \lb 0, 1 \rb^4 \times \mbR \rp$, we have 
	\bequ\label{equ:inveq}
		\mbE\lb e^{-\la \eta, f \ra}\rb = \mbE\lb e^{-\la \eta, f_t \ra}\rb, \quad \forall t > 0,
	\eequ
	where $f_t$ is defined as
	\bequ\label{f_t}
		f_t(x, h) = -\log \mbE\lb e^{-f\lp x, h + W_t - \frac{\pi t}{2} \rp} \rb, \quad x \in \lb 0, 1 \rb^4, h \in \mbR,
	\eequ
	and the expectation is taken w.r.t.\ the standard Brownian motion $W_t$.
	%# TODO: check here whether the definition needs the substraction of $\pi t/2$
\end{Thm}
This theorem contains the most lengthy proof of this work, which uses both the tightness results and a framework  outlined by \cite{biskup2016extreme} to 
prove the invariance under Dysonization. The tightness results rely on a delicate analysis
of the tail behavior of several maxima summation statistics of the 4D membrane model which is obtained in \Cref{sec:mBRW} by 
the comparison theorem between 4D membrane model and the modified branching random walk (MBRW). Moreover, the characteristic of the geometry of the field also comes
into the proof of the Dysonization invariance property.

Then we use this invariant property under ``Dysonization'' to extract the Poisson law that the limiting extremal process will satisfy:
\begin{Thm}\label{thm:ppp}
	Suppose $\eta$ is a point process on $[0, 1]^4 \times \mbR$ such that \eqref{equ:inveq} holds for some $t > 0$ 
	and all continuous $f : [0, 1]^4 \times \mbR \rightarrow [0, \infty)$ with compact support. Assume also that 
	$\eta([0, 1]^4 \times [0, \infty)) < \infty$ and $\eta([0, 1]^4 \times \mbR) > 0$ a.s. Then there is a 
	random Borel measure $Z$ on $[0, 1]^4$, satisfying $Z ([0, 1]^4) \in (0, \infty)$ a.s., such that
	\begin{equation}
		\eta \overset{\text{law}}{=} \PPP\lp Z\lp dx \rp \otimes e^{-\pi h} dh \rp.
	\end{equation}
\end{Thm} 
The proof follows the same routine as in 
\cite{biskup2016extreme}, which applies the technique developed by Liggett \cite{liggett1978random} on interacting particle systems and 
invariant measure of Markov chains. So the proof is omitted here. For 
readers interested in the argument, we refer to Theorem 3.2 and Section 3.2 in \cite{biskup2016extreme}. 

Thanks to Theorem \ref{thm:ppp}, the uniqueness of the Borel measure $Z$ is the only remaining piece for Theorem \ref{thm:Poisson-limit}. 
To prove this, we follow the framework outlined by \cite{biskup2016extreme} to prove the uniqueness of its Laplace transformation.
This amounts to the construction of an auxiliary field associated with the 4D membrane field and analyze its covariance structure.

Among all the ingredients of the proof, the proof of \Cref{thm:ppp} is relatively standard compared with \Cref{thm:4Dm-geo} and \Cref{thm:dysoninv}. Compared with the GFF case, 
the proof of the extremal process of 4D membrane model is more intricate due to the more complicated covariance structure while the overall strategy shares the same structure. Therefore
our main contribution lies in handling the detailed covariance structure of the 4D membrane model and the geometry of the near maxima.

These current results provide the first step toward a complete understanding of the extremal process of 4D membrane models. %Apart from the existence, %the uniqueness of the limit process and 
The exact characterization of the extremal processes remain open for this model. In comparison, the conformal symmetry and 
the freezing behavior of the extremal process was established in \cite{biskup2014conformal} and \cite{biskup2018full} respectively for GFF. Moreover, %given that the membrane model has dimension 4 as its critical dimension, 
it is also %would be 
very natural to consider its relation to the uniform spanning forests and Ising 
models in %which also have 
(the critical) dimension 4. Some relation between the biharmonic field and uniform spanning forest has already been established in \cite{sun2013uniform}.

The paper is organized as follows: we review some established results on 4D membrane models in \Cref{sec:literature}. To precisely estimate the extreme statistics and extremal processes of 4D membrane model, we leverage the previously introduced modified branching random walk (MBRW) for finer approximation of the 4D membrane model. We combine the estimate of the MBRW along with comparison theorems of Gaussian processes to obtain control on the 4D membrane models in \Cref{sec:mBRW} and present the complete 
proof in \Cref{sec:proof}. % by dividing into four parts.

\section{Basic facts and established results on 4D membrane model}\label{sec:literature}

In this section, we review some basic facts and properties of 4D membrane models. %We attempt to make our argument as detailed and complete as possible.
%\subsection{Basic definitions and properties}

One of the most important properties that both membrane models and GFF enjoy is the Gibbs-Markov property, also known as the domain Markov property. This is a cornerstone for us to characterize the geometry of near maxima of these models and establish the limit law of the extremal process. For Gibbs-Markov property in general setting, we suggest \cite{georgii2011gibbs} for a detailed exposition.
\begin{Thm}[Gibbs-Markov property of (4D) membrane model \cite{li2022intermediate}]
	The (4D) membrane models satisfy the Gibbs-Markov property. Specifically, for any $U \subsetneq V \subsetneq \mbZ^d$, one has the following decomposition
	\bequ\label{G-M}	
		\begin{aligned}
		h_u^{V} & = \mbE \lb h_u^{V} \Big| h_v^{V}, v \in V \backslash U \rb + h_u^{U, V}, \quad h_u^{U, V} \overset{\text{law}}{=} h_u^{U}, \\ h_u^{U, V} & \perp \mbE \lb h_u^{V} \Big| h_v^{V}, v \in V \backslash U \rb,
		\end{aligned}
	\eequ
	which is exactly the $L^2$ decomposition of conditional expectation. The $h_u^{U, V}$ part has the same law as the (4D) membrane models on the domain $U$. 
\end{Thm}
\begin{remark}
By Gaussianity of the membrane model, the conditional expectation of the field in $U$, i.e. $\mbE \lb h_u^{V} \Big| h_v^{V}, v \in V \backslash U \rb$ is a linear combination of the field outside $U$, i.e., $h_v^{V}, v \in V \backslash U$. Moreover, since the bi-Laplacian operator has range $2$, one further has reduction $$\mbE \lb h_u^{V} \Big| h_v^{V}, v \in V \backslash U \rb =  \mbE \lb h_u^{V} \Big| h_v^{V}, v \in \p_2 U \rb,$$ where $\p_k U := \lbb x \in U^c, \dist\lp x, U \rp \leq k \rbb.$ Here the distance refers to graph distance. Note that (4D) membrane models differs from the GFF in the part that Laplacian operator for GFF has range $1$, thus one has  $$\mbE \lb h_u^{\nabla, V} \Big| h_v^{\nabla, V}, v \in V \backslash U \rb =  \mbE \lb h_u^{\nabla, V} \Big| h_v^{\nabla, V}, v \in \p_1 U \rb.$$
\end{remark}
We refer to \cite{kurt2009maximum} for some comments on the relation between the Gibbs-Markov decomposition property, the Green's function, and the boundary condition for the corresponding Biharmonic equation. 

%\subsection{Early work on membrane models}

Similar to GFF, the membrane model is defined in any dimension. 4D is important for this model due to its criticality, corresponding to the fact that 2D is the critical dimension of GFF. This is characterized by the 
the following proposition on the covariance of 4D membrane models, indicating that it falls in the regime of log-correlated field.
% TODO: Check Runsheng's paper and other papers whether they stated the results
\begin{proposition}[Covariance of 4D membrane models \cite{schweiger2020maximum}]\label{prop:4D-cov}
	The covariance of the 4D membrane models $h_v$ satisfies
	\begin{itemize}
		\item
		There exists constant $\alpha_0'$ such that for $u, v \in V_N$, one has
		\bequ
			\begin{aligned}
				\Var\lp h_v^{N} \rp & \leq \frac{8}{\pi^2}\log N + \alpha_0; 		\\
				\mbE\lb \lp h_u^{N} - h_v^{N} \rp^2 \rb & \leq \frac{8}{\pi^2}\lp 2\log \lp 1 + \lv u - v \rv \rp - \lv \Var\lp h_v^{N} \rp - \Var\lp h_u^{N} \rp \rv + 4\alpha_0 \rp.
			\end{aligned}
		\eequ
		\item
		For any $\delta > 0$ there is a constant $\alpha^{\delta} > 0$ such that for all $u, v \in V_N$ with $\min\lp d_N\lp u \rp, d_N\lp v \rp \rp \geq \delta N$, one has
		\bequ\label{far-bound}
			\lv \Cov\lp h_v^{N}, h_u^{N}\rp - \log\frac{N}{1 + \lv u - v \rv} \rv \leq \alpha^{\delta}.
		\eequ 
	Here	$d_N\lp v \rp := \dist\lp v, \p V_N \rp$
		\item 
		There are a constant $\theta_0 > 0$, a continuous function $f_1: \lp 0, 1 \rp^4
		\rightarrow \mbR $ and a function $f_2: \mbZ^4 \times \mbZ^4 \rightarrow \mbR$
		such that the following holds. For all $L, \epsilon > 0, \theta > \theta_0$ there
		exists $N_0' = N_0'\lp L, \epsilon, \theta \rp$ such that for all
		$x \in \lb 0, 1 \rb^4, N \geq N_0', Nx \in \mbZ^4$ and
		$d\lp x \rp \geq \frac{\lp \log N \rp^{\theta}}{N}$ and for all
		$u, v \in \lb 0, L \rb^4 \cap \mbZ^4$, %\rev{[WHY IS THERE AN $f_2$?]}\JX{This is copied from FLORIAN SCHWEIGER's paper.}
		\bequ
			\lv \Cov\lp h^{N}_{Nx + v}, h^{N}_{Nx + u} \rp - \frac{8}{\pi^2}\log N - f_1\lp x \rp - f_2\lp u, v \rp \rv < \epsilon.
		\eequ
		\item 
		There are a constant $\theta_1 > 0$ and a continuous function $f_3: D^4 \rightarrow \mbR$, where $D^4 := \lbb \lp x, y \rp, x, y \in \lp 0, 1 \rp^4, x \neq y \rbb$ such that the following holds: for all $L, \epsilon > 0, \theta > \theta_1$ there exists $N_1' = N_1'\lp L, \epsilon, \theta \rp$ such for all $x, y \in \lb 0, 1 \rb^4, Nx, Ny \in \mbZ^4$ and $d\lp x \rp \geq \frac{\lp \log N \rp^{\theta}}{N}$, $\min\lp d\lp x \rp, d\lp y \rp \rp \geq \frac{\lp \log N \rp^{\theta}}{N}$ and $\lv x - y \rv \geq \frac{1}{L}$ we have	
		\bequ\label{non-diag-G}
			\lv \Cov\lp  h^{N}_{Nx},  h^{N}_{Ny} \rp - f_3\lp x, y \rp \rv < \epsilon.
		\eequ
	\end{itemize}
\end{proposition}
Next, we recall a lemma in \cite{kurt2009maximum}.
\begin{proposition}[Lemma 2.11 in \cite{kurt2009maximum}]\label{prop:var-decomp}
	Letting $0 < n < N$, let $A_N \subset  \mbZ^d$ be a box of side-length $N$ $v_B \in \mbZ^d$ and $A_n \subset A_N$ be a concentric box of side-length $n$. Letting $0 < \epsilon < \half$, there exists $c > 0$ such that, for all $v \in A_n$ with $\lv v - v_B \rv \leq \epsilon n$,
	\bequn
		\Var\lp \mbE\lb h_v^{N} \Big| \mcF_{\p_2 A_n} \rb - \mbE\lb h_{v_B}^{N} \Big| \mcF_{\p_2 A_n} \rb \Big| \mcF_{\p_2 A_N} \rp \leq c\epsilon.
	\eequn
\end{proposition}

Denote the maximum of the field by $M_N:= \max_{v \in V_N} h_{v}$. %and its expectation by $ m_N= \mbE M_N$. Under this scaling, 
It was shown in \cite{schweiger2020maximum}that 
the expectation of the maximum of the field is of order 
\bequ
	m_N= \frac{8}{\pi} \log N - \frac{3}{2\pi}\log \log N,
\eequ
while the expectation of maximum of GFF is given by $m_N^{\nabla} = 2\sqrt{\frac{2}{\pi}}\log N - \frac{3}{4}\sqrt{\frac{2}{\pi}}\log \log N$.
In \cite{schweiger2020maximum}, the following theorem is established, which contains all the information we can say about the maximum of 4D membrane models currently.
\begin{Thm}[Convergence in law of the maximum of 4D membrane model\cite{schweiger2020maximum}]\label{thm:convergence-in-law}
The random variable $M_N- m_N$
%\bequ
 %= M_N- \frac{8}{\pi} \log N + \frac{3}{2\pi}\log \log N
%\eequ
converges in distribution and the limit law is a randomly shifted Gumbel distribution $\mu_{\infty}$, given by
\bequ
	\mu_{\infty}\lp \lp - \infty , x \rb \rp = \mbE \lb e^{-\gamma^* \mcZ e^{-\pi x}} \rb,
\eequ
where $\gamma^*$ is a constant and $\mcZ$ is a positive variable that is the limit in law of 
\bequ
	\mcZ_N = \sum_{v \in V_N} \frac{8\log N - \pi h_{v}^{N}}{\sqrt{8}} e^{\pi h_{v}^{N} - 8\log N}.
\eequ
\end{Thm}
This theorem is proved in \cite{schweiger2020maximum} by checking a general criterion introduced in \cite{ding2017convergence}.

%More recently, one of the authors and collaborators studied the scaling limit of the intermediate level sets, i.e. the set of points above a certain ratio of the maximum \cite{li2022intermediate}.
%\begin{Thm}
%	Define the rescaled point measure as
%	\bequn
%		\eta_N := \frac{1}{K_N} \Sigma_{x \in V_N} \delta_{x/N} \otimes \delta_{h_x - a_N},
%	\eequn
%	where $\lim_{N \rightarrow \infty}\frac{a_N}{\log N} := 2\lambda\sqrt{2\gamma}, K_N := \frac{N^4}{\sqrt{\log N}}c^{-\frac{a_N^2}{2\gamma\log N}}$. Then, there is random Borel measure $Z_{\lambda}$ on $V_N$ with $\mbE[Z_{\lambda}(\overline{V_N})] \in \mbR_+$ such that
%	\bequn
%		\eta_N \rightarrow Z_{\lambda}(dx) \otimes e^{-\pi \lambda h}dh
%	\eequn
%	to the topology of vague convergence of measures on $\overline{D} \times \mbR$.
%\end{Thm}

%%%%%%%%%%%%%%%%%%%%%%%%
\begin{comment}
\begin{Thm}
	
\end{Thm}
\begin{remark}
	Here is a list of several properties of the limiting PPP:
	\begin{itemize}
		\item $Z^D\lp \p D \rp = 0, a.s.$, $Z^D\lp A \rp = 0, a.s.$ if $Leb\lp A \rp = 0$.
		\item $Z^D\lp D \rp < \infty, a.s.$
	\end{itemize}
\end{remark}
\end{comment}
%%%%%%%%%%%%%%%%%%%%%%%%

\section{Comparisons with BRW}\label{sec:mBRW}
The comparison theorems between Gaussian processes are very powerful in studying the extreme statistics of the Gaussian processes. 
In the study of GFF, the comparison between GFF and branching random walk (BRW) is a key tool in the proof of the convergence of the 
extremal process. To establish the convergence of the extremal process of 4D membrane models, we follow the same spirit, and translate the 
bounds on the modified BRW (MBRW, as abbreviated in \Cref{sec:Mainresults}) to the bounds on the 4D membrane model via comparison theorems. 

We start with some discussions on the 4D MBRW.

\subsection{Some basics on 4D membrane models and 4D BRW}
Instead of considering 2D BRW 
which branches into 4 particles at each time in the GFF setting, we consider BRW on 4D branching into 16 particles. This only affects the properties of BRW by a constant. We characterize this in the following proposition. Before that, we also provide the definition of 4D BRW here. See related study for BRW in \cite{zeitouni2016branching}.
\begin{definition}[Definition of 4D BRW \& MBRW\cite{ding2014extreme}]\label{Def-BRW}
	Denote $N = 2^n$. For $k = 0, 1, \cdots, n$, let $\mcB_k$ denote the collection of subsets of $\mbZ^4$ consisting of 4-cubes of side-length $2^k$ with corners in $\mbZ^4$, and let $\mathcal{BD}_k$ denote the collection of subsets of $\mbZ^4$ consisting of 4-cubes of the form $\lp \lb 0, 2^k - 1 \rb \cap \mbZ \rp^4 + \lp i2^k, j2^k, m2^k, n2^k \rp$. Note that $\mathcal{BD}_k$ provide a disjoint partition of $\mbZ^4$. For $v \in V_N$, let $\mcB_k\lp v \rp(\mathcal{BD}_k\lp v \rp)$ denote the elements in $\mathcal{B}_k (\mathcal{BD}_k)$ that contain $v$. It is easy to check that $\forall v \in V_N, \mathcal{BD}_k\lp v \rp$ contains exactly one element.		
 
		Next, consider $\{ a_{k, B} \}_{B \in \mathcal{BD}_k}$ a sequence of i.i.d.\ family of standard Gaussian random variables. The BRW $\lp \mcR_{v}^N \rp_{v \in V_N}$ is defined as
		\bequ
			\theta_{v}^N = \sum_{i = 0}^n \sum_{B \in \mathcal{BD}_k\lp v \rp} a_{k, B}.
		\eequ 
		Meanwhile, we define $R_N = \max_{v \in V_N} \theta_v^N$ and $r_N = \mbE R_N.$ 	
  
		Next, consider $\{ a_{k, B} \}_{B \in \mathcal{B}_k}$ a sequence of i.i.d.\ family of Gaussian random variables with variance $2^{-4k}$. The MBRW $\lp \xi_{v}^N \rp_{v \in V_N}$ is defined as
		\bequ
			\xi_{v}^N = \sum_{k = 0}^n \sum_{B \in \mathcal{B}_k\lp v \rp} a_{k, B}.
		\eequ 
		%We define $\mathsf{M}_N = \max_{v \in V_N} \xi_{v}^N, \mathsf{m}_N = \mbE \mathsf{M}_N.$ 
        Note that one of the key differences between BRW and MBRW is that MBRW splits the covariance of each point at level $k$ into $2^{4k}$ parts.
\end{definition}
% Furthermore, we define the following quantity \rev{THIS CLASHES WITH THE EXPECTED MAX OF MEMBRANE MODEL}
% \bequ
% 	\tilde{m}_N = 2\sqrt{\frac{2}{\log 2}} \log N - \frac{3}{8}\sqrt{\frac{2}{\log 2}}\log \log N.
% \eequ
% Note that this is distinct from the $m_N$ defined earlier for the 4D membrane model expectation.
\begin{proposition}[Comparison of covariance between 4D membrane and 4D MBRW]\label{Cov-comp}
		There exists a constant $C$ s.t.\ $\forall u, v \in V_N$,
		\bequ
			\begin{aligned}
				\lv \Cov\lp h_{u + 2N \times \lp 1, 1, 1, 1 \rp}^{4N}, h_{v + 2N \times \lp 1, 1, 1, 1 \rp}^{4N} \rp - \frac{8\log 2}{\pi^2}\lp n - \log_2 d_{\infty}^N\lp u, v \rp \rp \rv \leq C		,		\\
				\lv \Cov\lp \xi_u^N, \xi_v^N \rp - \lp n - \log_2 d_{\infty}^N\lp u, v \rp \rp \rv \leq C,
			\end{aligned}
		\eequ
		where the distance is given by $d_{\infty}^N\lp u, v \rp = 
		\mathop{\min}_{u \sim_N v}\norml u - v \normr_{\infty}$.
	\end{proposition}
	\begin{proof}
		The proof is by direct computation. For 4D MBRW, we first define the following quantities associated with $u = \lp u_1, u_2, u_3, u_4 \rp, v = \lp v_1, v_2, v_3, v_4 \rp$
		\bequn
			t_i\lp u, v \rp = \min\lbb \lv u_i - v_i \rv, \lv u_i - v_i - N \rv, \lv u_i - v_i + N \rv \rbb.
		\eequn
		Then, one has
		\bequn
			\begin{aligned}
				\Cov\lp \xi_u^N, \xi_v^N \rp = & \ \sum_{k = \lceil \log_2 d_{\infty}^N\lp u, v \rp \rceil}^n 2^{-4k}\prod_{i = 1}^4\lp 2^k - t_i\lp u, v \rp \rp 		\\
				= & \ \sum_{k = \lceil \log_2 d_{\infty}^N\lp u, v \rp \rceil}^n \prod_{i = 1}^4\lp 1 - \frac{t_i\lp u, v \rp}{2^k} \rp 		
				\leq  \ n - \log_2 d_{\infty}^N\lp u, v \rp + 1.
			\end{aligned}
		\eequn
	%	where this upper bound is simply obtained by taking all the $t_i\lp u, v \rp = 0$. 
 The lower bound can be obtained via the inequality $\prod_{i = 1}^4\lp 1 - \epsilon_i \rp \geq 1 - \sum_{i = 1}^4 \epsilon_i$ for $\epsilon \geq 0$. One deduces
		\bequn
			\begin{aligned}
				& \ \sum_{k = \lceil \log_2 d_{\infty}^N\lp u, v \rp \rceil}^n \prod_{i = 1}^4\lp 1 - \frac{t_i\lp u, v \rp}{2^k} \rp  
				\geq  \ \sum_{k = \lceil \log_2 d_{\infty}^N\lp u, v \rp \rceil}^n \lp 1 - \sum_{i = 1}^4\frac{t_i\lp u, v \rp}{2^k} \rp			\\
				\geq & \ n - \log_2 d_{\infty}^N\lp u, v \rp - \sum_{k = \lceil \log_2 d_{\infty}^N\lp u, v \rp \rceil}^n \sum_{i = 1}^4\frac{t_i\lp u, v \rp}{2^k}		\\
				\geq  & \ n - \log_2 d_{\infty}^N\lp u, v \rp - 4d_{\infty}^N\lp u, v \rp\sum_{k = \lceil \log_2 d_{\infty}^N\lp u, v \rp \rceil}^n \frac{1}{2^k}		\\
				\geq & \ n - \log_2 d_{\infty}^N\lp u, v \rp - 4\frac{d_{\infty}^N\lp u, v \rp}{2^{\lceil \log_2 d_{\infty}^N\lp u, v \rp \rceil - 1}}		
				\geq  \ n - \log_2 d_{\infty}^N\lp u, v \rp - C.
			\end{aligned}
		\eequn
		The bound for 4D membrane models is obtained by leveraging the Lemma 1.4 of \cite{sun2013uniform} and Corollary 2.9 in \cite{kurt2009maximum}. Note that the factor $\log 2$ appearing 
		in the covariance of 4D membrane models also appears in that of 2D GFF. Since in the lattice the base number is $2$ instead of $e$.
	\end{proof}
	%\begin{remark}
%		Since all the norms on finite dimension Banach space are equivalent, choosing different norms will only change the covariance by a constant. This comparison holds for any norm we choose.
%	\end{remark}
 
	\subsection{Comparison of the maximal sum}
	In this subsection, we establish the comparison between 4D membrane models and BRWs.
	We define the following two auxiliary functionals associated with the random field which will help us study the geometry of near maxima.
	\bequn
		\begin{aligned}
			h_{N, r}^{\diamond} := & \ \max\lbb h_{u}^{N} + h_{v}^{N}: u, v \in V_N, r \leq \norml u - v \normr \leq \frac{N}{r} \rbb,		\\
			\wtd h_{N, r}^{\diamond} := & \ \max\lbb h_{u + 2N \times \lp 1, 1, 1, 1 \rp}^{4N} + h_{v + 2N \times \lp 1, 1, 1, 1 \rp}^{4N}: u, v \in V_N, r \leq \norml u - v \normr \leq \frac{N}{r} \rbb,		\\
			\xi_{N, r}^{\diamond} := & \ \max\lbb \xi_{u}^{N} + \xi_{v}^{N}: u, v \in V_N, r \leq \norml u - v \normr \leq \frac{N}{r} \rbb,		\\
		\end{aligned}
	\eequn
	Next, we prove the following proposition.
	\begin{proposition}\label{pair-cmp}
		There exists a constant $3 \leq \kappa \in \mbZ_+$ such that for all $r, n \geq \kappa$ positive integers, with $N = 2^n$, one has
		\bequ
			\sqrt{\frac{8\log 2}{\pi^2}}\mbE\lb \xi_{2^{-\kappa}N, r}^{\diamond} \rb \leq \mbE\lb h_{N, r}^{\diamond} \rb \leq \sqrt{\frac{8\log 2}{\pi^2}}\mbE\lb \xi_{2^{\kappa}N, r}^{\diamond} \rb.
		\eequ
	\end{proposition}
	\begin{proof}
	The proof is based on the Sudakov-Fernique comparison theorem for Gaussian processes which is guaranteed by a careful analysis of covariance.		\\
		
	We first recall two basic relations between $h_{N, r}^{\diamond} $ and $\wtd h_{N, r}^{\diamond}$, i.e.,
	\bequ\label{inequ1}
		\begin{aligned}
			\mbE\lb h_{N, r}^{\diamond} \rb & \ \leq \mbE\lb \wtd h_{N, r}^{\diamond} \rb,		\\
			\mbP\lp \max_{v \in V_N} h_v^{N} \geq \lambda \rp & \ \leq 2\mbP\lp \max_{v \in V_N} h_{v + 2N \times \lp 1, 1, 1, 1 \rp}^{N} \geq \lambda \rp
		\end{aligned}
	\eequ
	These are proved in \cite{ding2014extreme} by the Gibbs-Markov property for the GFF, which translates easily to the case of 4D membrane models. Hence the proof is omitted.
	
	Now, we turn back to the desired conclusion. We first prove the upper bound, which is guaranteed by 
	$\mbE\lb \wtd h_{N, r}^{\diamond} \rb \leq \sqrt{\frac{8\log 2}{\pi^2}}\mbE\lb \xi_{2^{\kappa}N, r}^{\diamond} \rb$. Since the first field is defined on the 4-cube with side length $4N$ while the second one 
	has side length $2^{\kappa}N$. To leverage the Sudakov-Fernique comparison theorem, we have to specify a pairing between these two fields, 
	which is chosen as $\Psi_N\lp u \rp = 2^{\kappa - 3}N \times \lp 1, 1, 1, 1 \rp + 2^{\kappa - 2}u, u \in V_N$. Note that for this map, 
	we view the first field as a field on $V_N$ instead of $V_{4N}$. This is an injective map from $V_N$ to $V_{2^{\kappa}N}$ and thus 
	induces a pairing between these two fields. Then by Gaussian comparison, it suffices to derive the following bound
	\bequ\label{equ:cond1}
		\begin{aligned}
			& \ \mbE\lb \lp h_{u + 2N \times \lp 1, 1, 1, 1 \rp}^{4N} + h_{v + 2N \times \lp 1, 1, 1, 1 \rp}^{4N} - h_{u' + 2N \times \lp 1, 1, 1, 1 \rp}^{4N} - h_{v' + 2N \times \lp 1, 1, 1, 1 \rp}^{4N}\rp^2 \rb 		\\
			\leq & \  \frac{8\log 2}{\pi^2}\mbE\lb \lp  \xi_{u}^{2^{\kappa}N} + \xi_{v}^{2^{\kappa}N} - \xi_{u'}^{2^{\kappa}N} - \xi_{v'}^{2^{\kappa}N} \rp^2 \rb, \qquad \forall u, v, u', v'.
		\end{aligned}
	\eequ
        We can expand the LHS and RHS as
	%%%%%%%%%%%%%%%%%%%%%%%%
	\begin{comment}
		\begin{align}
			LHS = & \  \frac{8\log 2}{\pi^2}\lp 4n \rpt & + & \ 2\lp n - c\lp u, v \rp \rp + 2\lp n - c\lp u', v' \rp \rp - 2\lp n - c\lp u, v' \rp \rp\nonumber \\
			&  & - & \left. 2\lp n - c\lp u', v \rp \rp - 2\lp n - c\lp u, u' \rp \rp - 2\lp n - c\lp v', v \rp \rp \rp	\nonumber	\\
			= & \  \frac{16\log 2}{\pi^2}\lp  \rpt & & \left.  c\lp u, v' \rp + c\lp u', v \rp + c\lp u, u' \rp + c\lp v', v \rp - c\lp u, v \rp - c\lp u', v' \rp \rp\nonumber \\ \nonumber
		\end{align}
\bequn
\begin{array}{cccc}
LHS = & \  \frac{8\log 2}{\pi^2}\lp 4n \rpt & + & 2\lp n - c\lp u, v \rp \rp + 2\lp n - c\lp u', v' \rp \rp - 2\lp n - c\lp u, v' \rp \rp\nonumber		\\
&  & - & 2\lp n - c\lp u', v \rp \rp - \left.   2\lp n - c\lp u, u' \rp \rp - 2\lp n - c\lp v', v \rp \rp \rp	\nonumber	\\
= & \  \frac{16\log 2}{\pi^2}\lp  \rpt & & \left.  c\lp u, v' \rp + c\lp u', v \rp + c\lp u, u' \rp + c\lp v', v \rp - c\lp u, v \rp - c\lp u', v' \rp \rp\nonumber \\ \nonumber
\end{array} 
\eequn
	\end{comment}
		%%%%%%%%%%%%%%%%%%%%%%%%
		\bequn
			\begin{aligned}
				LHS = & \  \frac{8\log 2}{\pi^2} \lb 4n \rpt + \left. 2\lp n - c\lp u, v \rp \rp + 2\lp n - c\lp u', v' \rp \rp - 2\lp n - c\lp u, v' \rp \rp \rb  \nonumber \\
				& \Phantom{\frac{8\log 2}{\pi^2} \lb 4n \rpt} - \left. 2\lp n - c\lp u', v \rp \rp - 2\lp n - c\lp u, u' \rp \rp - 2\lp n - c\lp v', v \rp \rp	\rb + O\lp 1 \rp \nonumber	\\
				= & \  \frac{16\log 2}{\pi^2}\lb  c\lp u, v' \rp + c\lp u', v \rp + c\lp u, u' \rp + c\lp v', v \rp - c\lp u, v \rp - c\lp u', v' \rp \rb + O\lp 1 \rp, \\ \nonumber	\end{aligned}
                \eequn
                and
                \bequn
                \begin{aligned}
				RHS = & \  \frac{8\log 2}{\pi^2} \lb 4\lp n + \kappa \rp \rpt + \left. 2\lp n + \kappa - c\lp u, v \rp - \kappa \rp + 2\lp n + \kappa - c\lp u', v' \rp - \kappa \rp  \rb  \nonumber \\
				& \Phantom{\frac{8\log 2}{\pi^2} \lb 4\lp n + \kappa \rp \rpt} \left. - 2\lp n + \kappa - c\lp u, v' \rp - \kappa \rp - 2\lp n + \kappa - c\lp u', v \rp - \kappa \rp \rpt \nonumber	\\
				& \Phantom{\frac{8\log 2}{\pi^2} \lb 4\lp n + \kappa \rp \rpt} \left. - 2\lp n + \kappa - c\lp u, u' \rp - \kappa \rp - 2\lp n + \kappa - c\lp v', v \rp - \kappa \rp	\rb + O\lp 1 \rp \nonumber	\\
				= & \  \frac{16\log 2}{\pi^2}\lb  c\lp u, v' \rp + c\lp u', v \rp + c\lp u, u' \rp + c\lp v', v \rp - c\lp u, v \rp - c\lp u', v' \rp + 2\kappa \rb + O\lp 1 \rp, \\ \nonumber
			\end{aligned}
		\eequn
		where $c\lp u, v \rp = \log_2 d_{\infty}^N\lp u, v \rp$ in above derivation. Notice the $O\lp 1 \rp$ term appearing above only 
		depends on the constant $C$ in proposition \Cref{Cov-comp} rather than $N$. Consequently, we can choose $\kappa$ large enough so that the equation \eqref{equ:cond1} holds. Comparing the field 
		$$\lbb h_{u + 2N \times \lp 1, 1, 1, 1 \rp}^{4N} + h_{v + 2N \times \lp 1, 1, 1, 1 \rp}^{4N}: u, v \in V_N, r \leq \norml u - v \normr \leq \frac{N}{r} \rbb$$
        with 
		$$\Big\{ \xi_{\Psi_N\lp u \rp}^{2^{\kappa}N} + \xi_{\Psi_N\lp v \rp}^{2^{\kappa}N}: u, v \in V_N,  r \leq \norml u - v \normr \leq \frac{N}{r} \Big\}$$ which is a subset of $$\lbb \xi_{u}^{2^{\kappa}N} + \xi_{v}^{2^{\kappa}N}: u, v \in V_{2^{\kappa}N}, r \leq \norml u - v \normr \leq \frac{2^{\kappa}N}{r} \rbb$$ 
		where we use $r \leq \norml u - v \normr \leq \frac{N}{r} \Longrightarrow r \leq \norml \Psi_N\lp u \rp - \Psi_N\lp v \rp \normr \leq \frac{2^{\kappa}N}{r}$, we obtain the right part of conclusion, i.e., 
		$\mbE\lb h_{N, r}^{\diamond} \rb \leq \sqrt{\frac{8\log 2}{\pi^2}}\mbE\lb \xi_{2^{\kappa}N, r}^{\diamond} \rb$.
		For the lower bound, the method is identical. One only has to modify the pairing between two fields. Namely, in this case, the range of 4D MBRW is $V_{2^{-\kappa}N}$ is smaller than the range for 4D 
		membrane models and one has to specify a map from $V_{2^{-\kappa}N}$ to $V_{N}$, which is now chosen to be $\Psi'_N\lp u \rp = 2^{\kappa}u, u \in V_N$ and it is easy to verify that the covariance 
		comparison holds between fields $$\Big\{ \xi_{u}^{2^{-\kappa}N} + \xi_{v}^{2^{-\kappa}N}: u, v \in V_{2^{-\kappa}N}, r \leq \norml u - v \normr \leq \frac{2^{-\kappa}N}{r} \Big\},$$ 
		$$\lbb h_{\Psi'_N\lp u \rp + 2N \times \lp 1, 1, 1, 1 \rp}^{4N} + h_{\Psi'_N\lp v \rp + 2N \times \lp 1, 1, 1, 1 \rp}^{4N}: u, v \in V_{2^{-\kappa}N}, r \leq \norml u - v \normr \leq \frac{2^{-\kappa}N}{r} \rbb$$ $$\subset \lbb h_{u + 2N \times \lp 1, 1, 1, 1 \rp}^{4N} + h_{v + 2N \times \lp 1, 1, 1, 1 \rp}^{4N}: u, v \in V_N, r \leq \norml u - v \normr \leq \frac{N}{r} \rbb.$$ 
		And the left part of conclusion is established, i.e., $\sqrt{\frac{8\log 2}{\pi^2}}\mbE\lb \xi_{2^{-\kappa}N, r}^{\diamond} \rb \leq \mbE\lb h_{N, r}^{\diamond} \rb$.
	\end{proof}
	The above proposition is one of the most important ingredients to derive the geometry of near maxima of a 4D membrane model. It bridges the gap between these models and 4D MBRW and transfers the information of 4D MBRW models to 4D membrane. The next subsection is devoted to some conclusions on the two models.
	%\subsection{Comparison of the maxima of sums of particles}
	
	We further compare the sums of more particles between 4D membranes and BRW. Firstly, we define the following quantities
	\bequ\label{def:several-maxima}
			\begin{aligned}
				S_{\ell, N} &:= \max_{v_1, v_2, \cdots, v_{\ell} \in V_N}\sum_{i = 1}^{\ell} h_{v_i}^{N}, 		\\
				R_{\ell, N} &:= \frac{2\sqrt{2\log 2}}{\pi}\max_{v_1, v_2, \cdots, v_{\ell} \in V_N}\sum_{i = 1}^{\ell} \theta_{v_i}^N, 		\\
				U_{\ell, N} &:= \arg\max_{v_1, v_2, \cdots, v_{\ell} \in V_N}\sum_{i = 1}^{\ell} h_{v_i}^{N}.
			\end{aligned}
	\eequ
	We prove the following proposition, which generalized the upper bound in a previous result to multiple summations.
	\begin{proposition}
		There exists an absolute constant $\kappa \in \mbN$ s.t. $\mbE S_{\ell, N} \leq \mbE R_{\ell, N2^{\kappa}}.$  
	\end{proposition}
	\begin{proof}	
		Consider $\wtd \theta_{v}^N = \theta_{v}^N + \kappa X_v$ where $X_v$'s are i.i.d.\ Gaussian r.v.\ and the corresponding quantity is denoted 
		by $\wtd R_{\ell, N}$. By the summation representation of BRW, it is clear to conclude $\mbE\wtd R_{\ell, N} \leq \mbE R_{\ell, N2^{\kappa}}$. 
		Now, let $X$ be another independent standard Gaussian r.v. and choose a sequence ${a_v, v \in V_N}$ such that
		\bequn
			\Var\lp h_{v + 2N \times \lp 1, 1, 1, 1 \rp}^{4N} + a_v X \rp = \frac{8\log 2}{\pi^2}\Var \lp \wtd \theta_{v}^N \rp, \quad v \in V_N.
		\eequn
		By \Cref{Cov-comp}, i.e., the variance between $h_{v + 2N \times \lp 1, 1, 1, 1 \rp}^{4N}$ and $\theta_{v}^N$ is bounded, we can 
		choose $\kappa$ large enough to guarantee the existence of $a_v$. Denote by $\wtd S_{\ell, N}$ the corresponding maximum sum of the field $h_{v}^{N} + a_v X$, 
		then via a similar argument of \eqref{inequ1}, we conclude that $\mbE S_{\ell, N} \leq \mbE \wtd S_{\ell, N}$. Lastly, it remains to prove 
		$\mbE \wtd S_{\ell, N} \leq \mbE\wtd R_{\ell, N}$. This is proved by leveraging a variant of Slepian's inequality, i.e., Lemma 2.7 in 	\cite{ding2014extreme} and combining previous equality on variances of two Gaussian processes and the following inequality on the covariance 
		structure
		\bequn
			\begin{aligned}
				\frac{8\log 2}{\pi^2} \mbE \lb \wtd \theta_{u}^N \wtd \theta_{v}^N \rb = \frac{8\log 2}{\pi^2} \mbE \lb \theta_{u}^N \theta_{v}^N \rb \leq & \ \mbE \lb \lp h_{u + 2N \times \lp 1, 1, 1, 1 \rp}^{4N} + a_u X \rp\lp h_{v + 2N \times \lp 1, 1, 1, 1 \rp}^{4N} + a_v X \rp\rb \\
				= & \  \mbE \lb h_{u + 2N \times \lp 1, 1, 1, 1 \rp}^{4N} h_{v + 2N \times \lp 1, 1, 1, 1 \rp}^{4N} \rb + a_ua_v.
			\end{aligned}
		\eequn
		This inequality is satisfied if we choose a sufficiently large $\kappa$.% large enough, which will make all the $a_v$ large, hence realizing the above inequality.
	\end{proof}
	An upper bound on the sum of several maxima of the 4D membrane model is a direct corollary of the above comparison between 4D membrane models and MBRW
	and the corresponding bound on MBRW.
	\begin{corollary}\label{cor:several-maxima-exp}
		For some constant $c$, $\mbE S_{\ell, N} \leq \ell(m_N - c\log \ell)$.
	\end{corollary}
 
	\subsection{Extreme statistics of MBRW}
	
	In this section, we recall several results on extreme statistics of MBRW which will be useful when we derive the corresponding results in 4D membrane models using comparisons of Gaussian processes. 
	Most existing works of BRW are based on the setting that the expectation of the number of branching each time is $2$, c.f. 
	\cite{mckean1975application, bramson1978maximal, lalley1987conditional, arguin2011genealogy, arguin2013extremal}. In the study of GFF, BRW as 
	a 2D field is defined to compare with GFF, which branches into $4$ new particles each time. In our case, since we have to compare with 
	the 4D membrane model, 4D BRW and MBRW have to be defined which branch into $16$ particles at each time as in the \Cref{Def-BRW} before. 
	Similarly, we have to slightly modify classical results to fit this setting. The proof of these results can be found in section 3 of 
	\cite{ding2014extreme}.
	\begin{proposition}[Lemma 3.2 in \cite{ding2014extreme}]
		Recall that $\lbb \theta_v^N \rbb_{v \in V_N}$ is 4D BRW. Then, one has
		\bequ
			\mbE\lb \max_{v \in V_N}\theta_v^N \rb = r_N = \wtd m_N + O\lp 1 \rp,
		\eequ
		where $\wtd m_N = 2\sqrt{\frac{2}{\log 2}} \log N - \frac{3}{8}\sqrt{\frac{2}{\log 2}}\log \log N$. Further, there exist constants $c, C > 0$, so that, for $y \in \lb 0, \sqrt{n} \rb,$
		\bequn
			ce^{-2\sqrt{2\log 2}y} \leq \mbP\lp \max_{v \in V_N}\theta_v^N \geq \wtd m_N + y \rp \leq C\lp 1 + y \rp e^{-2\sqrt{2\log 2}y}.
		\eequn
	\end{proposition}
	The next result is a crucial bound on the quantity $\xi_{N, r}^{\diamond}$.
	\begin{proposition}\label{prop:bound-pair-MBRW}
		There exist constants $c_1, c_2 > 0$ such that 
		\bequ
			2m_N - c_2 \log \log r \leq \mbE\xi_{N, r}^{\diamond} \leq 2m_N - c_1 \log \log r,
		\eequ
	\end{proposition}
	\begin{proof}
		The proof of the upper bound is similar to that in the comparison theorem of Gaussian processes. Let $S_v^N$ be a BRW of depth $N$ and set $R_v^N = (1-\epsilon_N)S_v^N - G_v$ where $G_V$ is a collection of 
		i.i.d. Gaussian r.v. with mean $0$ and variance $\sigma^2$ to be specified later and $\epsilon_N = O(1/n)$. We choose $\sigma, \epsilon_N$
		such that $\mbE\lp R_v^N\rp^2 = \mbE\lp \xi_v^N\rp^2$ and $\mbE\lp R_v^N - R_u^N \rp^2 = \mbE\lp \xi_v^N - \xi_u^N \rp^2$. 
		With the comparison theorem, we conclude the upper bound in \Cref{prop:bound-pair-MBRW}.

		To prove the lower bound, we will use a lemma stated below.
	\end{proof}
	\begin{lemma}
		There exist constants $C_1, C_2 > 0$ such that for all $N$ large and all $r$,
		\bequ
			\mbP\lp \tilde{m}_N - C_1 \log \log r \leq \wtd \xi_{N, r}^{\diamond} \rp \geq C_2.
		\eequ
	\end{lemma}
	Using \Cref{prop:bound-pair-MBRW} and previous comparison result \Cref{pair-cmp} between 4D membrane and 4D MBRW, one deduces
	\begin{corollary}\label{cor:bound-pair-4Dm}
		There exist constants $c_1, c_2, C > 0$ such that 
		\bequ
			2m_N- c_2 \log \log r - C \leq \mbE h_{N, r}^{\diamond} \leq 2m_N- c_1 \log \log r + C.
		\eequ
	\end{corollary}
	This will be the main inequality to derive the geometry of near maxima of 4D membrane models in the next section, i.e. \Cref{thm:4Dm-geo}.

	Finally, we state without proof the following estimate on the tail bound of the sum of several maxima of 4D membrane models which will be an important
	ingredients in the proof of the tightness results.
 % TODO: seems that we need to prove this result, consider this

 	Define 
	\begin{equation}\label{def:geo-set}
		\Xi_{N, r} = \{(u, v)\in V_N\times V_N: r\leq \lv u - v \rv \leq N/r \}.
	\end{equation}
	\begin{corollary}\label{cor:several-maxima-tail}
		We have
		\begin{equation}
			\mcP(\exists A \subset \Xi_{N, r}: \lv A \rv > \log r, (u, v)\in A \Rightarrow h_u + h_v > 2(m_N - \lambda\log\log r)) > 1 - Ce^{-e^{c\lambda\log\log r}}
		\end{equation}
		for some $C,c > 0$, all $N \geq 1$ and all $r,\lambda \geq C$.
	\end{corollary}
	\begin{comment}
	Now, we introduce a new quantity as below
	\bequn
	\wtd \xi_{N, r}^{\diamond} :=  \ \max\lbb \xi_{u}^{N} + \xi_{v}^{N}: u, v \in V_N, r \leq \norml u - v \normr \leq \frac{N}{r}, \dist\lp u, \p V_N \rp, \dist\lp v, \p V_N \rp \geq \frac{N}{4} \rbb.
	\eequn
	Namely, $\wtd \xi_{N, r}^{\diamond}$ differs from $\xi_{N, r}^{\diamond}$ in that it only considers the pair of points that are sufficiently far from the boundary.
	\begin{proposition}[proposition. 3.5. in \cite{ding2014extreme}]
		There exist constants $C_1, C_2 > 0$ such that for all $N$ large and all $r$,
		\bequ
			\mbP\lp 2m_N - C_1 \log \log r \leq \wtd \xi_{N, r}^{\diamond} \rp \geq C_2.
		\eequ
	\end{proposition}
\end{comment}

\section{Proof of the main results}\label{sec:proof}

In this section, we give the complete proof to characterize the extremal process of 4D membrane models. The proof is divided into four subsections. 
The two most essential subsections are sec. \Cref{sec:geometry} and \Cref{sec:invariance}, which deal with the geometry of near maxima and distributional invariance respectively. Furthermore, the 
tightness and uniqueness results are established in \Cref{sec:tight} and \Cref{sec:unique} respectively. 

\subsection{Geometry of the near maxima}\label{sec:geometry}

This subsection is devoted to the proof of \Cref{thm:4Dm-geo}. We gather all the tools in the previous section to prove it. Before that, we need the following lemma which is proved by the so-called sprinkling trick.
	\begin{lemma}[Sprinkling lemma]\label{lem:sprinkling}
	There exists a constant $C > 0$ such that, if
	\bequn
		\mbP\lp  \exists (v, u) \in \Xi_{N, r}, \   h_u^{N}, h_v^{N} \geq m_N- \lambda  \rp \geq \epsilon,
	\eequn
	for some $\epsilon, \lambda > 0$ and $N, r \in \mbZ_+$, then for any $\delta > 0$, setting $N'$ as the smallest power of $2$ larger than or equal to $C\lp \delta, \epsilon \rp N$ and $\gamma = C\lp \sqrt{\frac{C\lp \delta, \epsilon \rp}{\delta}} \rp$ with $C\lp \delta, \epsilon \rp = \frac{2\log \delta}{\log\lp 1 - \epsilon \rp}$, the following holds
	\bequn
		\mbP\lp h_{N', r}^{\diamond} \geq 2m_N- 2\lambda - \gamma \rp \geq 1 - \delta.
	\eequn
	\end{lemma}
	\begin{proof}
		The proof uses a kind of bootstrap method based on the Gibbs-Markov property of the field. Let $N' = 2^{k + 3}N$ where $k = \lceil \log_2C\lp \delta, \epsilon \rp - 3\rceil$. Let $V_i \subset V_{N'}, i = 1, 2, \cdots, 2^k$ be $2^k$ disjoint 4-cubes in $V_{N'}$ with side length $8N$. Let $V_i^* \subset V_i, i = 1, 2, \cdots, 2^k$ be $2^k$ 4-cubes in $V_{i}$ respectively with side length $N$. We require that $V_i, V_i^*$ are concentric for all $i$. 
		
		Firstly, let $h_{v}^{V_i} = h_{v}^{V_i^*} + \mbE\lb h_{v}^{V_i} \Big| h_{v}^{V_i}, v \in V_i \backslash V_i^* \rb$ be the Gibbs-Markov decomposition of the 4D membrane models on $V_i^* \subset V_i$. Using the assumption and the fact that $\mbE\lb h_{v}^{V_i} \Big| h_{v}^{V_i}, v \in V_i \backslash V_i^* \rb$ is a centered Gaussian process, one has (recall the definition of $\Xi_{N, r}$ in \Cref{def:geo-set})
		\bequn
			\begin{aligned}
				\mbP\lp  \exists v, u \in V_i: r \leq \lv v - u \rv \leq \frac{N}{r}, \   h_u^{V_i} + h_v^{V_i} \geq 2m_N- 2\lambda  \rp
				 \geq \frac{\epsilon}{2}.
			\end{aligned}
		\eequn
		
		For each cube $V_i, i = 1, 2, \cdots, 2^k$, let the Gibbs-Markov decomposition on $V_i \subset V_{N'}$ be
		\bequn
			h_{v}^{N'} = h_{v}^{V_i} + \mbE\lb h_{v}^{N'} \Big| h_{v}^{N'}, v \in V_{N' }\backslash V_i \rb,
		\eequn 
		here we directly denote the first term as 4D membrane models on $V_{8N} = V_i$. Then, we have that for each $i = 1, 2, \cdots, 2^k$
		\bequn
			\mbP\lp  \exists v, u \in V_i: r \leq \lv v - u \rv \leq \frac{N}{r}, \   h_u^{V_i}, h_v^{V_i} \geq m_N- \lambda  \rp \geq \frac{\epsilon}{2}.
		\eequn
		In each cube $V_i$, denote $\lp \tau_{i}^1, \tau_{i}^2 \rp = \arg\max_{v, u \in V_i, r \leq \lv v - u \rv \leq \frac{N}{r}} \lbb h_u^{V_i} + h_v^{V_i} \rbb$. Therefore, one has 
		\bequn
			\mbP\lp  h_{\tau_{i}^1}^{V_i} + h_{\tau_{i}^2}^{V_i} \geq 2m_N- 2\lambda  \rp \geq \frac{\epsilon}{2}.
		\eequn
		Now, by the conditional independency between $V_i$, we conclude that
		\bequn
			\mbP\lp  h_{\tau^1}^{V_i} + h_{\tau^2}^{V_i} \geq 2m_N- 2\lambda  \rp \geq 1 - \lp 1 - \frac{\epsilon}{2} \rp^{2^k},
		\eequn
		where $ \lp \tau^1, \tau^2 \rp= \arg\max_{1 \leq i \leq 2^k} \lbb h_{\tau_i^1}^{V_i} + h_{\tau_i^2}^{V_i} \rbb $. Now, return to the Gibbs-Markov decomposition at these two points, one has that
		\bequn
		\begin{aligned}
				& \ \Var\lp \mbE\lb h_{\tau^j}^{N'} \Big| h_{\tau^j}^{N'}, v \in V_{N' }\backslash V_i \rb \rp = \Var\lp h_{\tau^j}^{N'} \rp - \Var\lp h_{\tau^j}^{V_i} \rp \\
				= & \  \frac{8}{\pi^2}\log N' - \frac{8}{\pi^2}\log N + O\lp 1 \rp = \frac{8}{\pi^2}\log \frac{N'}{N} + O\lp 1 \rp = O\lp k \rp, \quad j = 1, 2.
		\end{aligned}
		\eequn
		where we use \Cref{prop:4D-cov}, \Cref{far-bound}, and the fact that $\tau^1, \tau^2$ are both far from the boundary in either field $h_{v}^{N'}$ or $h_{v}^{V_i}$. Therefore, by Markov inequality,
		one has
		\bequn
			\begin{aligned}
				& \ \mbP\lp  h_{\tau^1}^{N'} + h_{\tau^2}^{N'} \geq 2m_N- 2\lambda - \gamma  \rp 		\\
				\geq & \  \mbP\lp \mbE\lb h_{\tau^1}^{N'} \Big| h_{\tau^1}^{N'}, v \in V_{N' }\backslash V_i \rb + \mbE\lb h_{\tau^2}^{N'} \Big| h_{\tau^2}^{N'}, v \in V_{N' }\backslash V_i \rb \geq - \gamma \rp 	\\& \cdot	 \ \mbP\lp  h_{\tau^1}^{V_i} + h_{\tau^2}^{V_i} \geq 2m_N- 2\lambda  \rp 	\\
				\geq & \ \lp 1 - \lp 1 - \frac{\epsilon}{2} \rp^{2^k} \rp \lp 1 - \frac{O\lp k \rp}{\gamma^2} \rp \geq 1 - \delta,
			\end{aligned}
		\eequn
		by our assumption of these constants.
	\end{proof}
	
	\begin{remark}
		Our proof of the sprinkling lemma follows the general philosophy of \cite{ding2014extreme}, but significantly streamlines a technical component of the original argument. Specifically, while their approach relies heavily on potential theory and a random walk representation to analyze the variance of the conditional field $\mathbb{E}\left[ h_{\tau^j}^{N'} ,\big|, h_{\tau^j}^{N'},, v \in V_{N'} \setminus V_i \right]$, we compute this variance explicitly.
	\end{remark}
    
	Next, we state several lemmas, which are exactly parallel to those in \cite{ding2014extreme}. The first lemma is extremely similar to lemma 4.5 in it, so we omit its proof. 
	\begin{lemma}
  The sequence of random variables $\lbb \frac{h_{N, r}^{\diamond} - \mbE \lb h_{N, r}^{\diamond} \rb}{\log \log r}\rbb$
	is tight for $r \leq N$.
 	\end{lemma}

	The next theorem also uses the idea of sprinkling to establish the result. Recall the definition of $\Xi_{N, r}$ in \Cref{def:geo-set}.
 	\begin{proposition}
		There exist absolute constants $C, c > 0$ such that for all $N \in \mbZ_+$ and $\lambda, r \geq C$
	\bequn	
		\mbP\lp \exists A \subset \Xi_{N, r}, \lv A \rv \geq \log r: \forall \lp u, v \rp \in A: h_{u}^{N} + h_{v}^{N} \geq 2m_N- 2\lambda \log \log r \rp \geq 1 - Ce^{-\sqrt{\lambda} \log\log r}.
	\eequn
	\end{proposition}
	\begin{proof}
	
	Denote by $R = N \lp \log r \rp^{-\frac{\sqrt{\lambda}}{100}}, l = N \lp \log r \rp^{-\frac{\lambda}{100}}$. Assume that the left bottom corner of $V_N$ is $\lp 0, 0, 0, 0 \rp$. Define $o_i = \lp il, 2R \rp, i = 1, 2, \cdots, M = \lfloor \frac{N}{2l} \rfloor = \frac{\lp \log r \rp^{\frac{\lambda}{100}}}{2}$. Let $C_i$ be a discrete ball of radius $r$ centered at $o_i$ and let $B_i \subseteq C_i$ be a box of side length $\frac{R}{8}$ centered at $o_i$. We next regroup the $M$ boxes into $m$ blocks. Let $\mcC_j = \lbb C_i: \lp j - 1 \rp m < i < jm \rbb$ and $\mcB_j = \lbb B_i: \lp j - 1 \rp m < i < jm \rbb$.
	
	Now we consider the maximal sum over pairs of the GFF in each $\mcB_j$. For ease of notation, we fix $j = 1$ and write $\mcB = \mcB_1$ and $\mcC = \mcC_1$. For each $B \in \mcB$, write the Gibbs-Markov decomposition 
	of $C$ as $h_{v}^{N} = h_{v}^{C} + \phi_v^{C}, v \in B$ and define $\lp \tau_{1, B}, \tau_{1, B} \rp = \arg\max_{\lp u, v \rp \in \lp B \times B \rp \cap \Xi_{N, r}}\lbb h_{u}^{C} + h_{v}^{C}\rbb$. Since $\lambda$ 
	is great enough, by \Cref{cor:bound-pair-4Dm} and the tightness of $\lbb \frac{h_{N, r}^{\diamond} - \mbE \lb h_{N, r}^{\diamond} \rb}{\log \log r}\rbb$, we have that
	\bequn
		\mbP\lp h_{\tau_{1, B}}^{C} + h_{\tau_{2, B}}^{C} \geq 2m_N- 2\lambda \log \log r \rp \geq  \frac{1}{4}.
	\eequn
	
Let $W = \lbb \lp \tau_{1, B}, \tau_{1, B} \rp : g_{\tau_{1, B}}^1 \geq  2m_N- 2\lambda \log \log r  , B \in \mcB \rbb$. 
By independence, a standard concentration argument gives that for an absolute constant 
$c > 0, \mbP\lp \lv W \rv \leq \frac{m}{8} \rp \leq e^{-cm}.$ It remains to study the process 
$\lbb \phi_u^{C} + \phi_v^{C} : \lp u, v \rp \in W \rbb.$ If $\phi_u^{C} + \phi_v^{C} \geq 0$, we have that $h_{\tau_{1, B}}^{N} + h_{\tau_{2, B}}^{N} \geq 2m_N- 2\lambda \log \log r$. 
By next lemma, we have that $$\mbP\lp \phi_u^{C} + \phi_v^{C} \leq 0, \forall \lp u, v \rp \in W \rp \leq Ce^{-\sqrt{\lambda} \log\log r}.$$
		
		Consequently, we have 
		\bequn
			\mbP\lp \max_{B \in \mcB }h_{\tau_{1, B}}^{N} + h_{\tau_{2, B}}^{N} \geq 2m_N- 2\lambda \log \log r \rp \geq 1 - Ce^{-\sqrt{\lambda} \log\log r}.
		\eequn
		We denote 
		\bequn
			\lp \tau_{1, j}, \tau_{1, j} \rp = \arg\max\lbb  \max_{B \in \mcB_j }h_{\tau_{1, B}}^{N} + h_{\tau_{2, B}}^{N} \rbb,
		\eequn
		and let $A := \lp \lp \tau_{1, j}, \tau_{1, j} \rp, j = 1, 2, \cdots, \frac{M}{m} \rp$. It is simple to conclude that A satisfies the condition of proposition with probability at least $1 - Ce^{-\sqrt{\lambda} \log\log r}.$
	\end{proof}
	The following is a key lemma in the previous proposition. It simplifies a tricky calculation based on potential theory in \cite{ding2013exponential} and then uses Slepian's comparison theorem to obtain the desired results.
	\begin{lemma}
		Let $U \subseteq \bigcup_{B \in \mcB} \lp B \times B \rp$ such that $\lv U \cap \lp B \times B \rp \rv \leq 1, \forall B \in \mcB$. Assume $\lv U \rv \geq \frac{m}{8}$,
		\bequn
			\mbP\lp \phi_u^{C} + \phi_v^{C} \leq 0, \forall \lp u, v \rp \in W \rp \leq Ce^{-\sqrt{\lambda} \log\log r}.
		\eequn
	\end{lemma}
	\begin{proof}
		The philosophy parallels that in \cite{ding2013exponential}, which establishes covariance estimates on $\phi_u^{C} + \phi_v^{C}, B \in \mcB$ to use Slepian's comparison theorem. Here, we simply provide a calculation for the covariance between $\phi_u^{C}, \phi_u^{C'}, B \neq B', B, B' \in \mcB$. Now, we write the Gibbs-Markov decomposition for the domain $C \cup C' \subset V_N$ on two points $u \in B \subset C, v \in B' \subset C'$ as 
		\bequn
			\begin{aligned}
			h_{u}^{N} & = h_{u}^{C \cup C'} + \phi_u^{C \cup C'};\qquad
			h_{v}^{N} & = h_{v}^{C \cup C'} + \phi_v^{C \cup C'}.
			\end{aligned}
		\eequn
		By the fact that $\dist\lp C, C' \rp \gg 1$, we conclude that above decomposition coincides with Gibbs-Markov decomposition on $C, C'$ individually, i.e.
		\bequn
			\begin{aligned}
			h_{u}^{N} & = h_{u}^{C} + \phi_u^{C};\qquad
			h_{v}^{N} & = h_{v}^{ C'} + \phi_v^{ C'}.
			\end{aligned}
		\eequn
		More specifically, there is a coupling of $\lp h_{u}^{C}, \phi_u^{C}, h_{v}^{ C'}, \phi_v^{ C'} \rp \overset{\text{law}}{\sim} \lp h_{u}^{C \cup C'}, \phi_u^{C \cup C'},\rpt$  $\left. h_{v}^{C \cup C'}, \phi_v^{C \cup C'} \rp$. Hence their covariance structure coincides. Since $C, C'$ are disconnected, we have that $h_{u}^{C \cup C'} \perp h_{v}^{C \cup C'}, u \in C, v \in C'$. And $h_{u}^{C \cup C'} \perp \phi_v^{C \cup C'}, h_{v}^{C \cup C'} \perp \phi_u^{C \cup C'}$ follows from the property of Gibbs-Markov decomposition. We can derive as follows
		\bequn
			\begin{aligned}
				\Cov\lp h_{u}^{N}, 		
			h_{v}^{N}\rp = & \ \Cov\lp h_{u}^{C \cup C'} + \phi_u^{C \cup C'},  h_{v}^{C \cup C'} + \phi_v^{C \cup C'} \rp		\\
			= & \ \Cov\lp h_{u}^{C \cup C'}, h_{v}^{C \cup C'} \rp + \Cov\lp h_{u}^{C \cup C'}, \phi_v^{C \cup C'} \rp \\
			& \ + \Cov\lp \phi_u^{C \cup C'}, h_{v}^{C \cup C'}\rp + \Cov\lp \phi_u^{C \cup C'}, \phi_v^{C \cup C'} \rp		\\
			= & \ \Cov\lp \phi_u^{C \cup C'}, \phi_v^{C \cup C'} \rp 		
			=  \ \Cov\lp \phi_u^{C}, \phi_v^{C'} \rp.
			\end{aligned}
		\eequn
		Hence, by standard result on the covariance structure of 4D membrane models, we have that
		\bequn
			\Cov\lp \phi_u^{C}, \phi_v^{C'} \rp = \Cov\lp h_{u}^{N}, 		
			h_{v}^{N}\rp = \log \frac{N}{l} = \frac{\lambda}{100}\log \log r = O\lp \frac{1}{\sqrt{\lambda}} \rp \Var\lp \phi_u^{C}\rp.
		\eequn
		And the Slepian's lemma applies directly to this situation when $\lambda$ is great.
	\end{proof}
 	\begin{corollary}\label{cor:left-tail}
 	There exist absolute constants $C, c > 0$ such that for all $N \in \mbZ_+$ and $\lambda, r \geq C$
	\bequn	
		\mbP\lp h_{N, r_k}^{\diamond} \geq 2m_N- 2\lambda \log \log r \rp \geq 1 - Ce^{-ce^{c\lambda \log\log r}}.
	\eequn
 	\end{corollary}

	Now, we are ready to prove the main theorem.
    
	\begin{proof}[Proof of the \Cref{thm:4Dm-geo}]
		Now, we can prove the main theorem for this subsection. Suppose that on the contrary, there exists a subsequence $\lbb r_k \rbb$ tends to infinity s.t.\ $\forall k$,
		\bequn
			\begin{aligned}
		\limsup_{N \rightarrow \infty} \mbP\lp  \exists (v, u) \in \Xi_{N,r_k},  h_u^{N}, h_v^{N} \geq m_N- c\log\log r_k  \rp \geq \epsilon, 
			\end{aligned}
		\eequn
		where we choose $c = \frac{c_1}{8}$ and $c_1$ is the constant appearing in  \Cref{cor:bound-pair-4Dm}. Consequently, by \Cref{lem:sprinkling}, for small $\delta > 0$ to be specified and $C\lp \epsilon, \delta \rp > 0$, we have  
		\bequn
			\begin{aligned}
		\limsup_{N \rightarrow \infty} \mbP& \lp  \exists v, u \in V_{C\lp \epsilon, \delta \rp N}: r_k \leq \lv v - u \rv \leq \frac{C\lp \epsilon, \delta \rp N}{r_k}, \rpt   \\
		& \qquad  \qquad \qquad \qquad \left. h_u^{C\lp \epsilon, \delta \rp N}, h_v^{C\lp \epsilon, \delta \rp N} \geq m_N- c\log\log r_k - C\lp \epsilon, \delta \rp \rp \geq 1 - \delta. 
			\end{aligned}
		\eequn
		Now, we consider random variables $W_{N, k} = 2m_N- 2c\log\log r_k - C\lp \epsilon, \delta \rp - h_{C\lp \epsilon, \delta \rp N, r_k}^{\diamond}$. By the preceding inequality, for and $\delta > 0$, there exists an integer $N_{\delta}$ such that $\mbP\lp W_{N, k} \geq 0 \rp \leq 2\delta$ for all $N \geq N_{\delta}, k \in \mbZ_+$.
		
		 Therefore, for $N \geq N_{\delta}, r_k \geq \max\lp e^e, C^* \rp$, we obtain that
		\bequn
			\begin{aligned}
				\mbE W_{N, k} \leq & \ A_{c, c^*, C^*}\delta \log\log r_k.
			\end{aligned}
		\eequn
		One obtains that 
		\bequn
			\mbE h_{C\lp \epsilon, \delta \rp N, r_k}^{\diamond} \geq 2m_{N} - \frac{c_1}{2} \log \log r_k - C\lp \epsilon, \delta \rp,
		\eequn
		which contradicts \Cref{cor:bound-pair-4Dm}.
	\end{proof}

\subsection{Tightness of the sequence of random measure}\label{sec:tight}

%Two crucial estimates on the level set of the 4D membrane models guarantee the tightness of the extremal process. % and enable the 
%extraction of a subsequence from it. The tightness result will be stated and proved in this subsection and the local geometry 
%will be studied in the next subsection.
We need an enhanced version of such tightness in \cite{biskup2016extreme} as below, e.g.\ summable w.r.t.\ $\lambda$, 
relying on the results in \cite{ding2014extreme}. For $\lambda > 0$, let $A_{N, \lambda}= \{ v \in V_N: h_v^{N} \geq m_N- \lambda \}$
be the set above the $\lambda$-level set.
		\begin{proposition}\label{prop:tight}
			There exists $\beta > 0$ such that for all $\lambda > 1$ and all large-enough $\kappa > 0$, 
			\bequn
				\sup_N \mbP\lp e^{\kappa \lambda} \leq \lv A_{N, \lambda}\rv \rp \leq e^{-\beta\kappa \lambda}.
			\eequn
			In particular, every weak subsequential limit $\eta$ of the processes $\lbb \eta_{N_k, r_k}\rbb$ if exists, must satisfy
			\bequn
				\mbP\lp e^{\kappa \lambda} \leq  \eta\lp \lb 0, 1 \rb^4 \times \lb -\lambda, \infty \rp \rp \rp \leq e^{-\beta\kappa \lambda}.
			\eequn
			for all $\lambda > 1$ and all large-enough $\kappa > 0$.
		\end{proposition}
		\begin{remark}
			Indeed, we can also prove a lower bound for the size of the level set as follows: there exist absolute constants $c, C$ s.t.
		\bequ
			\begin{aligned}
				\lim_{\lambda \rightarrow \infty}\lim_{N \rightarrow \infty} \mbP\lp ce^{c\lambda} \leq \lv A_{N, \lambda}\rv \leq Ce^{C\lambda}  \rp = 1.
			\end{aligned}
		\eequ
		This was originally proved as Theorem 1.2 in \cite{ding2014extreme} for 2D GFF which shared a lot of similarities with the proof of the 
		\Cref{thm:4Dm-geo}, and we leave the proof of 4D membrane case to interested readers. \Cref{prop:tight} is an enhanced version
		of the upper tail in this theorem, and we observe that this enhanced upper tail is enough to prove all the results.
		\end{remark}
	
    The proof of this proposition relies on the following two lemmas, which serve as analogues of Lemmas 4.2 and 4.3 in \cite{biskup2016extreme}.
The first lemma provides an estimate on the sum of several of the largest values in the four-dimensional membrane model, as defined in \eqref{def:several-maxima}. While our formulation is adapted to the membrane model, the core argument follows the same spirit as that in \cite{biskup2016extreme}, hence we omit the proof for brevity.

		\begin{lemma}\label{lem:tight1}
			There are constants $c_1, c_2 \in (0, \infty)$ such that for all $\lambda > 0$, $\ell \geq 1$ and $N \geq 1$,
			\bequn
				\mbP(S_{\ell, N} < \ell(m_N- \lambda)) > c_1 - c_2 \frac{\lambda}{\log \ell}.
			\eequn
		\end{lemma}
The second lemma controls the size of the level set $A_{N, \lambda}$ once we have \Cref{lem:tight1}.
		\begin{lemma}\label{lem:tight2}
			There are constants $c_3, c_4 \in (0, \infty)$ such that the following is true for all $\epsilon, \delta > 0$ that are sufficiently small: 
			Set $\sigma := -\frac{\log \delta}{\epsilon}$, given $\lambda > 0$ and $\ell \in \mbN$, and define
			\bequn
				\lambda' := \lambda + c_4(1 + \log \sigma +  \sqrt{\log(\ell/\delta)\log \sigma}).
			\eequn
			If for some $N' \geq 1$, 
			\bequn
				\mbP(S_{\ell, N'} < \ell(m_N- \lambda')) > \delta,
			\eequn
			then 
			\bequn
				\mbP(|A_{N, \lambda}| > \ell ) < \epsilon
			\eequn
			holds for all $N \in \mbN$ with $N \leq c_3 \sigma^{-1/2}N'.$
		\end{lemma}
		\begin{proof}
			We will prove the contrary. Fix $\epsilon, \delta > 0$ small enough and suppose that for some $\lambda > 0, l \geq 1$ and $N \geq 1$, 
			\bequn
				\mbP(|A_{N, \lambda}| > \ell ) \geq \epsilon.
			\eequn
			Note that this implies 
			\bequn
				\mbP(S_{\ell, N} > \ell(m_N- \lambda)) \geq \epsilon.
			\eequn
			As the first step, consider a box $V_N$ of side $N$ contained in, and centered at the same point near the center of, a box $V_{N''}$ of side $N''$. If $h''$ is a GFF on $V_{N''}$ with zero boundary conditions, then we claim that
			\bequn
				\mbP(S_{\ell, N}(h'') > \ell(m_N- \lambda)) \geq \frac{\epsilon}{2},
			\eequn
			note that $h''$ is defined on $V_{N''}$ but $S_{\ell, N}(h'')$ refers to its restriction on $V_N$. To show this, recall that according to Gibbs-Markov decomposition, one has $h'' \overset{\text{law}}{=} h + \phi''$, where $\phi''$ is a mean-zero Gaussian process which is independent with $h$. Consequently, one concludes that 
			\bequn
				\lbb S_{\ell, N}(h'') > \ell(m_N- \lambda) \rbb \supseteq \lbb S_{\ell, N}(h) > \ell(m_N- \lambda)) \rbb \cap \Big\{ \sum_{x \in U_{l, N}(h)}\phi_x'' \geq 0 \Big\}.
			\eequn
			And the above relation provides us with the desirable inequality. Now, by sprinkling lemma, we conclude that for arbitrary $\delta$ we can choose $N'$, s.t.
			\bequn
				\mbP(S_{\ell, N'} > \ell(m_N- \lambda)) \geq 1 - \delta,
			\eequn
			which contradicts the assumption of this lemma.	
		\end{proof}
		\begin{comment}
		\begin{proof}[Proof of \Cref{prop:tight}]
			The proposition holds the same conclusion as \Cref{lem:tight2} which uses \Cref{lem:tight1} as a key ingredient. 
			One needs to pay attention to the choice of different constants to match the desired results.
			Specificaly, given constants $c_1, c_2$ from \Cref{lem:tight1}. Given $\lambda, \kappa > 1$, 
			define $l := e^{\kappa\lambda}, \epsilon := e^{-\beta\kappa\lambda}$ where $\beta \in (0, 1)$ remains to 
			be determined. Then by direct calculation, one has
			\bequn
				\begin{aligned}
					\lambda' = & \ \lambda + c_4(1 + \log\log\frac{1}{\delta} + \beta\kappa\lambda + \sqrt{\log\log\frac{1}{\delta} + \beta\kappa\lambda}\sqrt{\kappa\lambda + \log\frac{1}{\delta}}) \\
					= & \ \lambda(1 + \wtd c_4\beta\kappa + \kappa\sqrt{\beta}) + o(\lambda) = \lambda(1 + c_4'\kappa\sqrt{\beta}) + o(\lambda) ,
				\end{aligned}
			\eequn
			By \Cref{lem:tight1}, we have for sufficiently small $\beta$ and sufficiently large $\lambda$
			\bequn
				\mbP(S_{\ell, N'} < \ell(m_N'- \lambda)) > c_1 - c_2 \frac{\lambda'}{\log \ell} = c_1 - c_2\frac{1 + c_4'\kappa\sqrt{\beta}}{\kappa} + o(1)> \delta,
			\eequn
			which combines with \Cref{lem:tight2} prove that $\mbP(|A_{N, \lambda}| > e^{\kappa\lambda} ) < e^{-\beta\kappa\lambda}$ for $N \leq c_3 \sigma^{-1/2}N'.$ Since $N'$ is arbitrary, we conclude the proof of this proposition.
		\end{proof}
		\end{comment}

\subsection{Distributional invariance}\label{sec:invariance}

In this subsection, we establish the other main ingredient in the proof of the convergence of the extremal 
process. Namely, this distributional invariance under ``Dysonization'' of its point by a simple diffusion 
uniquely characterizes the limiting distribution as a Poisson point process. We will employ the Gaussian interpolation method described in \cite{biskup2016extreme} in the case of 4D membrane models.

Let $h^{1}, h^{2}$ be independent copies of the 4D membrane models in $V_N$. Let $t > 0$ be a fixed number 
which does not change with $N$ and abbreviate
\bequn
	\hat{h}^{1} := \sqrt{1 - \frac{t}{g\log N}} h^{1}, \quad \hat{h}^{2} := \sqrt{\frac{t}{g\log N}} h^{2},
\eequn
where $g = \frac{8}{\pi^2}$. Note that this parameter is chosen so that we have $\Var\lp h_x^{1} \rp \sim g\log N$ with $x$ away from the boundary $\p V_N$. The key to the proof of the fact is that the interpolated field 
\bequ
	h:= \hat{h}^{1} + \hat{h}^{2}.
\eequ
also has the law of the 4D membrane models in $V_N$.

Now, recall the definition \eqref{equ:def-extremal-process} of the extremal process in 4D membrane model. We define the following set
\bequn
	\Theta_{N, r} := \lbb x \in V_N, \quad h_x= \max_{z \in \Lambda_r(x)} h_z\rbb,
\eequn
and we use $\Theta_{N, r}^1$ to denote the corresponding set associated with the field $h^1$. Clearly, for arbitrary function $f: [0, 1]^4 \times \mbR \rightarrow \mbR$,
\bequn
	\la \eta_{N, r}, f \ra = \sum_{x \in \Theta_{N, r}} f(\frac{x}{N}, h_x- m_N).
\eequn
If $f$ is compactly support, we can insist on $x \in A_{N, \lambda}$ for $\lambda$ large enough s.t. $f$ vanishes on $[0, 1]^4 \times (-\infty, -\lambda)$. %Recall that $A_{N, \lambda}$ denote the level set ``$\lambda$ units below $m_N$''.
Our first observation is that to consider the maximum, we have to choose a sequence $r_N \rightarrow \infty$, which does not differ a lot from using only one $r$.
\begin{lemma}\label{subst}
	Let $f: [0, 1]^4 \times \mbR \rightarrow \mbR$ be a measurable function with compact support. For any $r_N$ with $r_N \rightarrow \infty, \frac{r_N}{N} \rightarrow 0$, 
	\bequn
		\lim_{r\rightarrow\infty}\limsup_{N \rightarrow \infty} \mbP(\la \eta_{N, r}, f \ra \neq \la \eta_{N, r_N}, f \ra) = 0.
	\eequn
\end{lemma}
This lemma is proved by exactly the same method of lemma 4.4 of \cite{biskup2016extreme} using the local geometry of the level set and we omit the proof here. After focusing our attention on a fixed $r$, we attempt to replace the condition $\{ x \in \Theta_{N, r} \}$ by $\{x \in \Theta_{N, r}^1 \}$, which is defined for the 4D membrane field $h^{1}$. We also define $A_{N, \lambda}^{1}$ as the corresponding set for $h^{1}$.
\begin{proposition}\label{Ginterpolation}
	Let $f: [0, 1]^4 \times \mbR \rightarrow \mbR$ be a measurable function with compact support. For any $\epsilon > 0$, 
	\bequ\label{inner-comp}
		\lim_{r\rightarrow\infty}\limsup_{N \rightarrow \infty} \mbP\Bigg(\Big| \sum_{x \in \Theta_{N, r}} f(\frac{x}{N}, h_x- m_N) - \sum_{x \in \Theta_{N, r}^1} f(\frac{x}{N}, h_x- m_N)  \Big| > \epsilon\Bigg) = 0.
	\eequ
\end{proposition}
Our proof technique shows that the values of $f$ on $\Theta_{N, r}$, restricted to a proper level set of $h$, are controllably close to the values of $f$ on $\Theta_{N, r}^1$ restricted to a proper level set of $h^1$. We first state a relation between the level set of $h$ and $h^{1}$.
\begin{lemma}\label{levelset}
	We have
	\bequn
		\begin{aligned}
		\lim_{\lambda\rightarrow\infty}\liminf_{N \rightarrow \infty} \mbP\lp A_{N, \lambda}^{1} \subseteq A_{N, 2\lambda}\rp = 1, 			\\
			\lim_{\lambda\rightarrow\infty}\liminf_{N \rightarrow \infty} \mbP\lp A_{N, \lambda}\subseteq A_{N, 2\lambda}^{1} \rp = 1.
		\end{aligned}
	\eequn
\end{lemma}
\begin{proof}
	For the first equality, it suffices to consider the set $A_{N, \lambda}^{1} \backslash A_{N, 2\lambda}$. We obviously have the following inclusion relation
	\bequn
		A_{N, \lambda}^{1} \backslash A_{N, 2\lambda}\subseteq \lbb x \in V_N,  h_x^{1} \geq m_N- \lambda,  \hat{h}_x^{2} < c - \lambda\rbb,
	\eequn
	where $c = \frac{8t}{\pi g}$ is chosen to depend only on $t$ instead of $N, \lambda$ such that $$\sqrt{1 - \frac{t}{g\log N}} h^{1} \geq \lp 1 - \frac{t}{g\log N} \rp \lp m_N- \lambda \rp \geq m_N  - \frac{t}{g\log N} m_N- \lambda \geq m_N- \lambda - \frac{8t}{\pi g}.$$ %which only 
 Now, since $h_x^{1}, \hat{h}_x^{2}$ are independent and $\hat{h}_x^{2}$ is centered with variance only depends on $t$, i.e. $O(t) = \Var(\hat{h}_x^{2}) \leq C$. We have the following inequality
	\bequn
		\mbP\Big(\big\{ x \in V_N,  h_x^{1} \geq m_N- \lambda,  \hat{h}_x^{2} < c - \lambda\big\} \neq \varnothing \Big) \leq \mbP(|A_{N, \lambda}^{1}| > e^{C'\lambda}) + e^{C'\lambda}e^{-\frac{(\lambda - c)^2}{2C}},
	\eequn
	which is obtained by decomposing the event into $|A_{N, \lambda}^{1}| > e^{C'\lambda}$ and $|A_{N, \lambda}^{1}| < e^{C'\lambda}$ and bound the second one by the first moment using the independency between $h^1, \hat{h}^2$. Consequently, by combining \Cref{prop:tight} we conclude that the RHS tends to $0$ as $N, \lambda$ tend to infinity as long as we choose $C'$ sufficiently large.		\\
	
	Next, we prove the second inequality which is more intricate. The methodology is that we decompose the corresponding set into a sequence of subsets which can be bounded using the argument in the above paragraph. 
	We abbreviate $a_N = \sqrt{1 - \frac{t}{g\log N}}$. We again decompose the corresponding set into
	\bequn
		A_{N, \lambda}\backslash A_{N, 2\lambda}^{1} \subseteq \bigcup_{n \geq 2}\lbb x \in V_N,  h_x^{1} > m_N- (n + 1)\lambda,  \hat{h}_x^{2} \geq (a_Nn - 1) \lambda\rbb.
	\eequn
	The above decomposition is obtained by choosing the unique $n$ s.t. $m_N- n\lambda \geq h_x^{1} > m_N- (n + 1)\lambda$ and then 
	$h > m_N - \lambda$ gives us $\hat{h}_x^{2} \geq (a_Nn - 1) \lambda$. Let $A_n$ denotes the event in the above decomposition with 
	index $n$ that is non-empty and $B_n$ denotes the event $\{ |A_{N, \lambda}^{1}| > e^{\kappa n \lambda}\}$ with given $\kappa$. 
	By \Cref{prop:tight}, there is $\beta > 0$ such that for all $\kappa, \lambda$ large enough, we have $\mbP(B_n) \leq e^{-\beta\kappa n \lambda}$. 
	Using again that $h_x^{1}, \hat{h}_x^{2}$ are independent and $\hat{h}_x^{2}$ is centered with variance only depends on $t$, i.e. 
	$O(t) = \Var(\hat{h}_x^{2}) \leq C$. We conclude 
	\bequn
		\mbP(A_n) \leq \mbP(A_n \backslash B_{n + 1}) + \mbP(B_{n + 1}) \leq e^{-\beta\kappa(n + 1)\lambda} + e^{\kappa(n + 1)\lambda} e^{-\frac{(a_Nn - 1)^2\lambda^2}{2C}}.
	\eequn
	Putting all the things together we obtain that
	\bequn
		\mbP(A_{N, \lambda}\backslash A_{N, 2\lambda}^{1}  \neq \varnothing) \leq \sum_{n \geq 2} (e^{-\beta\kappa(n + 1)\lambda}+ e^{\kappa(n + 1)\lambda} e^{-\frac{(a_Nn - 1)^2\lambda^2}{2C}}),
	\eequn
	which tends to $0$ as $N, \lambda \rightarrow \infty, a_N \rightarrow 1$. We comment here that it is the infinite sum appearing in the above bound that requires a bound on $\mbP(|A_{N, \lambda}| > e^{C\lambda} )$ summable in $\lambda$.
\end{proof}
Obviously, the extreme local maxima of $h_x$ and $h_x^{1}$ will coincide only if the field $\hat{h}_x^{2}$ does not vary much in the neighborhood of these extreme points. This naturally leads to the study of the oscillation of $\hat{h}_x^{2}$, which is defined as $osc_{A}g := \max_{z \in A}g(z) - \min_{z \in A}g(z)$. We have the following lemma whose proof we omit as it is very similar to Lemma 4.7 \cite{biskup2016extreme}.
\begin{lemma}
	For any $\lambda > 0$, any $\delta > 0$ and any $r \geq 1$,
	\bequn
		\limsup_{N \rightarrow \infty} \mbP(\max_{x \in A_{N, \lambda}^{1}}osc_{\Lambda_{2r}(x)}\hat{h}_x^{2} > \delta) = 0.
	\eequn
\end{lemma}
Consider now the a.s.\ well-defined mappings:
\bequn
	\Pi(x) := \arg\max_{\Lambda_{2r}(x)} h, \quad \Pi^1(x) := \arg\max_{\Lambda_{2r}(x)} h^1.
\eequn
Our next claim deals with the closeness of $\Pi^1(\Theta_{N, r})$ to $\Theta_{N, r}^1$, and $\Pi^1(\Theta_{N, r}^1)$ to $\Theta_{N, r}$, provided these are restricted to proper level sets.
\begin{lemma}\label{noname}
	The following holds with probability tending to one in the limits $N \rightarrow \infty, \delta \downarrow 0, r \rightarrow \infty$, and $\lambda \rightarrow \infty$ (in this given order):
	\bequn
		x \in \Theta_{N, r} \cap A_{N, \lambda} \cap A_{N, \lambda}^{1} \Longrightarrow \left\{ 
		\begin{aligned}
			& \Pi^1(x) \in \Theta_{N, r}^1, \quad \lv \Pi^1(x) - x \rv \leq \frac{r}{2},	\\
			& 0 \leq h_x - h_{\Pi^1(x)} \leq \delta,
		\end{aligned}\right.
	\eequn 
	and
	\bequn
		x \in \Theta_{N, r}^1 \cap A_{N, \lambda} \cap A_{N, \lambda}^1 \Longrightarrow \left\{ 
		\begin{aligned}
			& \Pi(x) \in \Theta_{N, r}, \quad \lv \Pi(x) - x \rv \leq \frac{r}{2},	\\
			& 0 \leq h_{\Pi(x)}  - h_x \leq \delta.
		\end{aligned}\right.
	\eequn 
\end{lemma}

Now, we can prove the main theorem which characterizes the distributional invariance of the extremal process.
\begin{proof}[Proof of \Cref{thm:dysoninv}]
	Denote $x \in \Lambda_x\lp r \rp$ as the maximum of $h^{1}$ in this neighborhood $\Lambda_x\lp r \rp$. Consequently, we have that $h_x^{1} = 2\sqrt{2g}\log N + o\lp \log N \rp$. Next, we expand the sum for $h$ using Taylor expansion, i.e.
	\bequn
	h= h^{1} - \frac{t}{2g\log N}h^{1} + \sqrt{\frac{t}{g\log N}} h^{2} + o\lp 1 \rp.
	\eequn
	At each point $z \in \Lambda_x\lp r \rp$, we have
	\bequn
		\begin{aligned}
			h_z= & \ h_z^{1} - \frac{t}{2g\log N}h_z^{1} + \sqrt{\frac{t}{g\log N}} h_z^{2} + o\lp 1 \rp 		\\
			= & \ h_z^{1} - \frac{t}{2g\log N}h_x^{1} + \sqrt{\frac{t}{g\log N}} h_x^{2} + o\lp 1 \rp		\\
			= & \ h_z^{1} - t\sqrt{\frac{2}{g}} + W_t + o\lp 1 \rp = h_z^{1} - \frac{\pi t}{2} + W_t + o\lp 1 \rp,
		\end{aligned}
	\eequn
	where in the above derivation we use frequently the estimate $\frac{t}{2g\log N}\norml h_z^{i} - h_x^{i} \normr = o\lp 1 \rp, \forall z \in \Lambda_x\lp r \rp$. This is due to $z \in \Lambda_x\lp r \rp$, we have that 
	\bequn
		\Var\lp h_z^{i} - h_x^{i} \rp = \Var\lp h_z^{i} \rp + \Var\lp h_x^{i} \rp - 2\Cov\lp h_z^{i}, h_x^{i} \rp = O\lp 1 \rp,
	\eequn
	by the fourth asymptotic property \eqref{non-diag-G} of Green's function of 4D membrane models. The term $\frac{t}{2g\log N}h_x^{1} $ is reduced by the assumption that it is the maximum of $h^{1}$. Lastly, the term $W_t$ comes from the fact that the term $\sqrt{\frac{t}{g\log N}} h_x^{2}$ has variance $t + o\lp 1 \rp$. Moreover, for $x, z$ in two different clusters, one can show that $\lp \sqrt{\frac{t}{g\log N}} \rp^2 \Cov\lp h_z^{i} , h_x^{i} \rp$, indicating that different clusters have this independent r.v.

	Consider the set $\Delta_{N, \rho}^1(\lambda) := \lbb x \in V_{N, \rho}: \lv h_x^1 - m_N \rv \leq \lambda \rbb$. Since $f$ has compact support, the level sets of $h, h^1$ are related as in \Cref{levelset}. Using also the tightness of the maximum of $h^1$, we thus have
	\bequn
		\lim_{\lambda \rightarrow \infty} \limsup_{r \rightarrow \infty} \limsup_{\delta \downarrow 0}\limsup_{N \rightarrow \infty} \mbP\Big( \sum_{x \in \Theta_{N, r}^1 \\ \Delta_{N, \delta}^1(\lambda)} f(\frac{x}{N}, h_x - m_N) > \epsilon \Big) = 0.
	\eequn
	This permits us to focus only on the sum of $f(\frac{x}{N}, h_x - m_N)$ over $x \in \Theta_{N, r}^1 \cap \Delta_{N, \rho}^1(\lambda)$. As $f$ is uniformly continuous, we now replace $h_x$ by $h_x^1 - \frac{\alpha}{2}t + \hat{h}^{2}$ and control the difference of the sums arising from the interpolation between the two fields using the bounded convergence theorem and the tightness results.
	
	Using  \Cref{Ginterpolation}, we conclude that
	\bequn
		\begin{aligned}
			\lim_{\lambda \rightarrow \infty} \limsup_{r \rightarrow \infty} \limsup_{\delta \downarrow 0}\limsup_{N \rightarrow \infty} & \ \mbP\Bigg( \Big| \la \eta_{N, r}, f \ra \\
			 - & \sum_{x \in \Theta_{N, r}^1 \cap \Delta_{N, \delta}^1(\lambda)} f(\frac{x}{N}, h_x^1 - m_N  -\frac{\alpha}{2}t + \hat{h}_x^2) \Big| > \epsilon \Bigg) = 0,
		\end{aligned}
	\eequn
	holds for any $\epsilon > 0$.	\\
	
	We aim to compute the limit of $\mbE e^{-\la \eta_{N, r}, f \ra}$. For this, we replace $\la \eta_{N, r}, f \ra$ by the sum in the above equation and take the conditional expectation given $h^1$. For $h^1$ such that the event $A_{N, \lambda, r}^1$ occurs, we have for all $u, v \in \Theta_{N, r}^1 \cap \Delta_{N, \rho}^1(\lambda)$ with $u \neq v$,
	\bequn
		\Var(\hat{h}_u^2) = t + o(1), \quad \Cov(\hat{h}_u^2, \hat{h}_v^2) = o(1), \quad N \rightarrow \infty.
	\eequn
	Therefore, by a lemma, on $A_{N, \lambda, r}^1 \cap \lbb |A_{N, \lambda}^1| \leq M \rbb$ we have
	\bequn
		\begin{aligned}
			& \ \mbE\Bigg[ \exp\Big\{ \sum_{x \in \Theta_{N, r}^1 \cap \Delta_{N, \delta}^1(\lambda)} f(\frac{x}{N}, h_x^1 - m_N  -\frac{\alpha}{2}t + \hat{h}_x^2) \Big\} \Bigg| h^1 \Bigg]  			\\
			= & \  e^{o(1)}\wtd\mbE \exp\Big\{ \sum_{x \in \Theta_{N, r}^1 \cap \Delta_{N, \delta}^1(\lambda)} f(\frac{x}{N}, h_x^1 - m_N  -\frac{\alpha}{2}t + W_t^{(x)}) \Big\}  			\\
			= & \exp\Big\{ o(1) - \sum_{x \in \Theta_{N, r}^1 \cap \Delta_{N, \delta}^1(\lambda)} f_t(\frac{x}{N}, h_x^1 - m_N) \Big\}, \quad N \rightarrow \infty,
		\end{aligned}
	\eequn
	where the expectation $\wtd \mbE$ is with respect to a measure under which $\lbb W_t^{(x)}: x \in \Theta_{N, r}^1 \cap \Delta_{N, \delta}^1(\lambda) \rbb$ are i.i.d. Gaussians with mean $0$ and variance $t$, the function $f_t$ is as defined as \eqref{f_t} The $o(1)$ terms is random but tending to zero uniformly in $h^1$ under consideration.	
	
	Taking expectation also w.r.t.\ $h^1$ and noting that the function under expectation is at most one, Theorem \Cref{thm:4Dm-geo} and \Cref{prop:tight} and the Bounded Convergence Theorem thus ensure
	\bequn
		\begin{aligned}
			\lim_{\lambda \rightarrow \infty} \limsup_{r \rightarrow \infty} \limsup_{\delta \downarrow 0}\limsup_{N \rightarrow \infty} & \Bigg| \mbE e^{-\la \eta_{N, r}, f \ra} \\
			 - & \mbE \exp\Big\{ \sum_{x \in \Theta_{N, r}^1 \cap \Delta_{N, \delta}^1(\lambda)} f(\frac{x}{N}, h_x^1 - m_N  -\frac{\alpha}{2}t + \hat{h}_x^2)\Big\} \Bigg| = 0.
		\end{aligned}
	\eequn
	We would like to drop the restriction to $\Delta_{N, \delta}^1$ and interpret the $N \rightarrow \infty$ limit (along the desired subsequence) of the second term using the limit process $\eta$. Unfortunately, we can not just roll the argument backward for $f_t$ in place of $f$ because $f_t$ no longer has compact support. Notwithstanding, from the fact that $f$ does, we get
	\bequn
		f_t(x, h) \leq Ce^{-ch^2}.
	\eequn
	Therefore, on the event that that $\lbb |A_N^1(\theta)| < Ce^{C\theta}: \forall \theta > \lambda \rbb \cap \lbb h_x^1 \leq m_N + \lambda: \forall x \in V_N \rbb,$ whose probability tends to $1$ as $N \rightarrow \infty$ and $\lambda \rightarrow \infty$ thanks to \Cref{prop:tight} and the tightness of the centered maximum, we have 
	\bequn
		\Bigg| \sum_{x \in \Theta_{N, r}^1 \cap \Delta_{N, \delta}^1(\lambda)} f_t(\frac{x}{N}, h_x^1 - m_N) \Bigg| \leq C\sum_{n \geq 1} e^{C(\lambda + n) - c(\lambda + n - 1)^2}.
	\eequn
	This is summable in $n$ and tends to $0$ as $\lambda \rightarrow \infty$. Using also \Cref{subst} we arrive at 
	\bequn
		\limsup_{N \rightarrow \infty} \lv \mbE e^{-\la \eta_{N, r_N}, f \ra} - \mbE e^{-\la \eta_{N, r_N}, f_t \ra}\rv = 0.
	\eequn
	Therefore, taking the limit, we recover \eqref{equ:inveq} in the theorem.
\end{proof}

\subsection{Extracting a Poisson law and uniqueness}\label{sec:unique}

As we have mentioned in the introduction, \Cref{thm:ppp} follows from \Cref{thm:dysoninv}, and the proof is omitted. We conclude that all the subsequential limits of the process $\eta_N, r_N$ take the form in
\Cref{thm:ppp}; the last remaining issue is to establish uniqueness in Theorem \ref{thm:Poisson-limit}. 
The idea is to prove the uniqueness via the uniqueness of the Laplace transformation as below:
\begin{proposition}\label{thm:joint-unique}
	Let $(A_1, \ldots, A_m)$ be a collection of disjoint non-empty open subsets of $[0, 1]^4$. Then the law of $\lp \max\{ h_x: x \in V_N, x/N \in A_l \} - m_N \rp_{l=1}^m$ 
	converges weakly as $N\rightarrow \infty$.
\end{proposition}
\begin{comment}
	As a by-product, we will also give Z(dx) the interpretation of a derivative martingale.
\end{comment}

\Cref{thm:joint-unique} is a direct corollary of the following theorem and we prove it by a modification of the proof in \cite{biskup2017extrema}. 
\begin{proposition}\label{thm:aux}
	Set $b_K:=c\log\log K$ for some $c>0$ sufficiently small. For any $m\geq 1$, let $d^{(m)}$ denote the Levy metric on probability measures on $\mbR^m$. Then, for any $m\geq 1$ and any collection
	$\underline{A} = \lp A_1, \ldots, A_m\rp$ of non-empty open subsets of $[0, 1]^4$,
	\begin{equation}
		\lim_{\delta \downarrow 0}\limsup_{K\rightarrow \infty}\limsup_{N\rightarrow \infty}d^{(m)}(\mu_{N, \underline{A}}, \mu_{K, \delta, \underline{A}}) = 0.
	\end{equation}
\end{proposition}
%# TODO: add details in the proof
To prove this theorem, we will use a coupling between the extreme process and a collection of i.i.d.\ random variables, 
which is based on the Gibbs-Markov property of the membrane model. Let us briefly recall several key definitions made in 
\cite[Section 5.1]{biskup2016extreme}: $w_i^K, i=1,\cdots,K^4$ is the $1/K$-mesh of $[0,1]^4$ box; $V_N^{K, i}$($V_N^{K, \delta, i}$)
is the partition (near-partitiion grid) of $V_N$ given $K|N, \delta N/K \in \mbN$:
\begin{equation}
	V_N^{K, i} := Nw_i^K + V_{N/K}, \quad V_N^{K, \delta, i} := Nw_i^K + (N\delta /K, N(1-\delta)/K)^4 \cap \mbZ^4.
\end{equation}
Moreover, \cite{biskup2016extreme} passes into continuum limit and introduce continuous version of the above set
\begin{eqnarray}
	B^{K, i} := w_i^K + (0,1/K)^4, \quad B^{K, \delta, i} := w_i^K + (\delta/K,(1-\delta)/K)^4.
\end{eqnarray}

%\JX{We first introduce a collection of i.i.d. random variables $\lp Y_i^K: i = 1, \ldots, K^2\rp$ with the law of $Y_i^K$ being the law of the maximum of the GFF on $B_N^{K, i}$, which is a box of side length $N/K$ centered at $z_i^K$. We then define}

We introduce a membrane field $h_v^f$ that we refer
to as the \textit{fine field} and that is a membrane model with zero boundary on these sub-boxes - it is constructed by
subtracting from the original membrane model field its conditional expectation, given the sigma-algebra generated
by the membrane model on the boundary of these sub-boxes. The fine field has the advantage that, due to the Markov property of the membrane model, its values in disjoint boxes are independent. We define the coarse
field as the difference $h_v^c = h_v^N - h_v^f$. The coarse field is, of course, correlated over the whole box $V_N$, but it is relatively smooth; in fact, for fixed $K$ as $N \rightarrow \infty$, the field obtained by rescaling the
coarse field onto a box of side length $1$ in $\mbR^4$ converges to a limiting Gaussian field that possesses continuous sample paths on appropriate subsets of $[0, 1]^4$ (essentially, away from the boundaries of
the sub-boxes).

An important step in our approach is the computation of the tail probabilities of the maximum
of the fine field when restricted to a box of side length $N/K$, together with the computation
of the law of the location of the maximum (in the scale $N/K$). These computations are done
by building on the tail estimates derived previously and using a modified second-moment method.
Another important step is to show that the maximum of the GFF occurs only at points where the
fine field is atypically large. Once these two steps are established, we can describe the limit law
of the membrane model by an appropriate mixture of random variables whose distributions are determined by
the tail of the fine field. The mixture coefficients are determined by an (independent) percolation
pattern of potential locations of the maximum and by the limiting coarse field, i.e.
\begin{equation}
	h^c \overset{\text{law}, \ N\rightarrow \infty}{\Longrightarrow} \Phi_K(\cdot),
\end{equation}
on $B^{K, \delta} = \cup B^{K, \delta, i}.$

%Based on the Gibbs-Markov property of the membrane model, we introduce a coupling between the extreme process and a collection of i.i.d. random variables.

Abbreviate $h^*:=\max_{x\in V_N}h_x$. The proof of the existence of the weak limit of the centered maximum $h^*-m_N$
is based on a comparison with an auxiliary process as follows: for a sequence of number $b_K$ determined later
define Bernoulli r.v.s $\mcP_i^K$, positive r.v.s $Y_i^k$
\begin{equation}
	\begin{aligned}
		\mbP(\mcP_i^K = 1) := & \ C_*b_Ke^{-\sqrt{2\pi}b_K},		\\
		\mbP(Y_i^K \geq x) := & \ \frac{b_K+x}{b_K} e^{-\sqrt{2\pi}x},		\\
		\mbP(K(z_i^K - w_i^K) \in A) := & \ \int_A \psi(x)dx,
	\end{aligned}
\end{equation}
where $C_*$ ($\psi$) is a constant (probability density) related to the upper tail of the maximum,

Using these, we set
\begin{equation}
	G_i^{K, \delta} := \lbb\begin{aligned}
		& Y_i^K + b_K - 2\sqrt{\gamma}\log K + \Phi_K(z_i^k), \quad \mcP_i^K = 1, z_i^K \in B_N^{K, \delta, i},	\\
		& - 2\sqrt{\gamma}\log K, \qquad otherwise
	\end{aligned}\right.
\end{equation}
Denote $G_{K, \delta}^* := \max_{i=1}^{K^2}G_i^{K, \delta}$. Let $\mu_N$ denote the law of $h^*-m_N$ and $\mu_{K, \delta}$ stand
for the law of $G_{K, \delta}^*$. 
%# TODO: we need to prove this for 4D membrane model
Theorem 2.3 in [19] then asserts that for $b_K := c \log \log K$ with some small $c > 0$, as $N \rightarrow \infty, K \rightarrow \infty$ and $\delta \rightarrow 0$, the Levy distance between $\mu_N$ and
$\mu_{K, \delta}$ tends to $0$. This yields weak convergence of $h^*-m_N$ as the auxiliary process does
not depend on $N$. 

For our purposes, we need a generalized unique result that addresses the joint law of maxima over several (different) subsets of $[0, 1]^4$. To this end, for any non-empty
open $A \subset [0, 1]^4$ we first define $h_A^*$ as the maximum of the 4D membrane model in $\{x \in V_N : x/N \in A\}$. Then, given a collection $\underline{A} = (A_1, . . . , A_m)$ of non-empty open subsets of $[0, 1]^4$, we
use $\mu_{N, \underline{A}}$ to denote the joint law of $\{h_{A_l}^* - m_N\}_{l=1}^m$. For the above auxiliary process we similarly define
\begin{equation}
	G_{K, \delta, \underline{A}}^* := \max\{G_i^{K, \delta}: B^{K, i} \subset A\},
\end{equation}
and write $\mu_{K, \delta, \underline{A}}$ for the joint law of $\{G_{K, \delta, \underline{A_l}}^*\}_{l=1}^m$.

We first gather the following ingredients for \Cref{thm:aux}.
\begin{lemma}
    The sequence $h_A^* - m_N$ is tight.
\end{lemma}
\begin{proof}
	One can find $K, i$ s.t. $B^{K, i} \subset A$. Fixing $\delta \in (0,1)$ and letting $z_i \in V_N^{K, \delta, i}$ be the maximum of $h^f$ 
	in $V_N^{K, \delta, i}$, we have $h_{z_i}^f + h_{z_i}^c \leq h_A^* \leq h^*.$ The tightness of $h^*$ implies that
	\begin{equation}
		\lim_{t\rightarrow \infty} \limsup_{N\rightarrow \infty} \mbP(h_A^* - m_N > t) \leq \lim_{t\rightarrow \infty} \limsup_{N\rightarrow \infty} \mbP(h_A^* - m_N > t) = 0.
	\end{equation}
	given the tightness of the upper tail of $h_A^*.$

	Concerning the lower tail we note that, since $h^f$ restricted to $V^{K,i}$ is a membrane model on a translate of $V_{N'/K}$ 
	and due to \Cref{thm:convergence-in-law}, we have
	\begin{equation}\label{equ:equ1}
		\lim_{\delta \downarrow 0}\limsup_{t\rightarrow \infty} \limsup_{N\rightarrow \infty} \mbP(h_{z_i}^f - m_{N/K} < -t) = 0.
	\end{equation}
	At the same time, we have $\forall x \in V_N^{K, \delta, i}$:
	\begin{equation}
		\Var(h_x^c) = \Var(h_x) - \mbE(\Var(h_x|\mcF_{N, K})) \leq \gamma (\log N - \log(N'/K) ) + C \leq g\log K + C',
	\end{equation}
	where $C'$ depends on $\delta$. This follows from \Cref{prop:var-decomp} since $x$ is away from the boundary of both $V_N$ and
	$V_N^{K,\delta, i}$. Conditioning on $z_i$ and using the independence of $h^f$ and $h^c$ we have
	\begin{equation}\label{equ:equ2}
		\mbP(h_{z_i}^c < -t) \leq C\exp\lp -\frac{t^2}{2(\gamma\log K + C')}\rp,
	\end{equation}
	uniformly in $N$. Writing
	\begin{equation}
		\mbP(h_A^* - m_N < -t) \leq \mbP(h_{z_i}^f - m_{N/K} < -t/2 + 2\sqrt{\gamma}\log K) + \mbP(h_{z_i}^c < -t/2),
	\end{equation}
	and the results follow by \eqref{equ:equ1} and \eqref{equ:equ2}.
\end{proof}
Next, we present some auxiliary definitions as follows: given a non-empty open $A \subset [0,1]^4$, we define 
$V_{N, A}^{K, \delta} := \cup\{ V_N^{K, \delta, i}: B^{K, i} \subset A \}$ and 
$\Delta_{N, A}^{K, \delta} := \{ x\in V_N: x/N \in A \} \backslash V_{N, A}^{K, \delta}$. As $A$ is open, 
the Lebesgue measure of $A\backslash \cup \{B^{K, i}: B^{K, i} \subset A\}$ tends to $0$ as $K \rightarrow \infty$.
One also has 
\begin{equation}
	\lim_{\delta \downarrow 0}\limsup_{K\rightarrow \infty}\limsup_{N\rightarrow \infty}\frac{|\Delta_{N, A}^{K, \delta}|}{|V_N|} = 0.
\end{equation}
All the following three results can be similarly proved using technique for tightness.
\begin{proposition}
    Let $A$ be any open non-empty subset of $[0, 1]^4$. Then
    \begin{equation}
        \lim_{\delta \downarrow 0}\lim\sup_{K\rightarrow \infty}\lim\sup_{N\rightarrow \infty} \mbP\lp \max_{x\in V_{N, A}^{K, \delta}}h_x \neq h_A^*\rp = 0.
    \end{equation}
\end{proposition}
\begin{proposition}\label{prop:aux}
    Let $A \subset [0, 1]^4$ be non-empty and open. Let $z_i$ be such that we have $\max_{x\in V_{N}^{K, \delta, i}}h_x^f = h_{z_i}$ 
	and let $\overline{z}$ be such that $\max\{ h_{z_i}: B^{K, i}\subset A\} = h_{\overline{z}}$. Then for all $\epsilon, \delta$
    \begin{equation}
        \lim_{K\rightarrow \infty}\lim\sup_{N\rightarrow \infty} \mbP\lp \max_{x\in V_{N, A}^{K, \delta}}h_x \geq h_{\overline{z}} + \epsilon\rp = 0.
    \end{equation}
    Furthermore, there exists a sequence $(B_K)_{K\geq 1}$ with $B_K \rightarrow \infty$ as $K \rightarrow \infty$ such that
    \begin{equation}
        \lim_{K\rightarrow \infty}\lim\sup_{N\rightarrow \infty} \mbP\lp h_{\overline{z}}^f \leq m_{N/K} + B_K\rp = 0.
    \end{equation}
\end{proposition}
\begin{lemma}
    For $i=1,...,K^4$, let $z_i$ be as in \Cref{prop:aux} and $z_i' \in V_{N}^{K, \delta, i}$ be measurable with respect to $\{h_x^f: x \in V_N^{K, i}\}$ such that $\lv z_i - z_i'\rv = o(N/K)$ as $N,K \rightarrow \infty$. Then for any $A \subset [0,1]^4$ open and non-empty 
    \begin{equation}
        \lim_{K\rightarrow \infty}\lim\sup_{N\rightarrow \infty} d\lp \max\{h_{z_i}^f + h_{z_i}^c: B^{K, i}\subset A\}, \max\{h_{z_i}^f + h_{z_i'}^c: B^{K, i}\subset A\}\rp = 0.
    \end{equation}
    where $d$ is the Levy metric for $\mbR$-valued random variables.
\end{lemma}
This lemma shows that small changes in the coordinate of the coarse field do not affect the resulting maxima and the proof is the 
same as \cite{bramson2012tightness} and hence omitted here.

\begin{proof}[Proof of \Cref{thm:aux}]
	For each $A_1, A_2, \cdots, A_m$ define
	\begin{equation}
		\overline{h}_{A_l}^*:= \max\{ h_{z_i}^f + h_{z_i}^c: B^{K, i} \subset A, h_{z_i}^f > m_{N/K} + b_K \}.
	\end{equation}
	Then for all $l = 1,\cdots,m$, we have $\mbP(\lv \overline{h}_{A_l}^* - h_{A_l}^*\rv > \epsilon) < \epsilon$.
	Denoting by $\nu_{N, \underline{A}}^{K, \delta}$ the law of $\lp \overline{h}_{A_l}^* \rp_{l=1}^m$, this shows
	$d(\nu_{N, \underline{A}}, \nu_{N, \underline{A}}^{K, \delta}) < \epsilon.$
	
	Next we employ the same coupling of $h$ with $\lp \mcP_i^K, Y_i^K, z_i^K \rp_{i=1}^{K^2}$, as in the original proof.
	We have the law of 
	\begin{equation}
		\lp \max\{ b_K + Y_i^K + h_{z_i^K}^c - (m_N - m_{N/K}): \mcP_i^K=1, z_i^K\in V_N^{K, \delta, i}, B_N^{K, i}\subset A_l \} \rp_{l=1}^m,
	\end{equation}
	to be denoted by $\overline{\nu}_{N, \underline{A}}^{K, \delta}$, satisfies $d^{(m)}(\overline{\nu}_{N, \underline{A}}^{K, \delta}, \nu_{N, \underline{A}}^{K, \delta}) < \epsilon$,
	for all $K, N$ large enough. Using the convergence of $h^c$ to $\Phi_K$ as $N\rightarrow \infty$ and triangle inequality, 
	we conclude that 
	\begin{equation}
		\lim_{\delta \downarrow 0}\lim\sup_{K \rightarrow \infty}\lim\sup_{N \rightarrow \infty}d^{(m)}(\mu_{N, \underline{A}}, \mu_{K, \delta, \underline{A}}) = 0,
	\end{equation}
	as desired. Finally, since $\mu_{K, \delta, \underline{A}}$ does not depend on $N$, this implies that 
	$\mu_{N, \underline{A}}$ is Cauchy and hence converges weakly to a limit.
\end{proof}

Now, using \Cref{thm:ppp} and \Cref{thm:joint-unique}, we can prove the uniqueness of the law of the full extremal process.

\begin{proof}[Proof of \Cref{thm:Poisson-limit}]
The proof relies on the fact that a probability distribution function is uniquely determined by its Laplace transform under certain conditions. For some $r_N$ with $r_N\rightarrow \infty$ and $r_N/N \rightarrow 0$ as
$N \rightarrow \infty$, let $\eta$ be a sub-sequential limit of $\eta_{N, r_N}$. Let $Z(dx)$ be the random measure so that
\eqref{equ:poisson-limit} holds and let $h_A^*$ denote the maximum of $h_x$ over $x \in V_N$ with $x/N \in A$. Note that in light
of \Cref{thm:joint-unique}, $h_A^* - m_N$ is tight. Given any collection $(A_1, \cdots ,A_m)$ of disjoint non-empty open
subsets of $[0,1]^4$ such that $Z(\p A_l) = 0$ a.s.\ for each $l = 1, \cdots ,m$, we then have
\begin{equation}\label{equ:laplace-transform}
	\mbE\lb \exp\lp -\alpha^{-1}\sum_{l=1}^m e^{-\alpha t_l}Z(A_l)\rp \rb = \lim_{N \rightarrow \infty} \mbP\lp h_{A_l}^* - m_N \leq t_l, l=1, \cdots, m\rp,
\end{equation} 
for any $t_1, \cdots , t_m \in \mbR$. (The convergence in \Cref{thm:joint-unique} ensures this for a dense set of $t_l$s; the
continuity of the left-hand side then extends this to all of Rm.)
By \Cref{thm:joint-unique} again, the right-hand side of \eqref{equ:laplace-transform} is the same for all subsequences and so this
proves the uniqueness of the law of integrals with respect to $Z(dx)$ of all positive simple functions on
open sets $A$ with $Z(\p A) = 0$ a.s. Using standard arguments this can be extended to the class of all
continuous functions on $[0,1]^4$. Hence, the law of $Z(dx)$ is itself unique.

The a.s.\ finiteness of the total mass of $Z$ arises (via the arguments in \Cref{thm:ppp}) from the
tightness of the upper tail of $\eta$. The fact that $Z(A) > 0$ a.s.\ for any $A \subset [0,1]^4$ open is a consequence
of the fact that $\eta(A\times \mbR) > 0$ a.s.\ for all such $A$. This follows from the Gibbs-Markov
property of the 4D membrane model and the fact that $A$ contains an open square.
\end{proof} 

\begin{remark}
	The proofs of both \Cref{thm:joint-unique} and \Cref{thm:Poisson-limit} can be extended to arbitrary domains $D$ beyond the unit box $[0,1]^4$, following the philosophy of \cite{biskup2017extrema}.
Moreover, we believe that the limiting measure in the membrane model admits an interpretation in terms of a derivative martingale, although we leave a detailed treatment of this for future work to maintain focus and clarity.
\end{remark}

\begin{comment}
The uniqueness of the limiting intensity measure can be obtained by the same argument in \cite{biskup2017extrema}, i.e. we have the following proposition.
\begin{proposition}[Uniqueness of the intensity measure]
	Suppose that $N_k \rightarrow \infty$ and one has the following limit
	\bequ\label{equ:poisson-limit}
		\eta_{N_k, r_k}\rightarrow \PPP\lp Z\lp dx \rp \otimes e^{-\pi t} dt\rp.
	\eequ
	Then, for each $t \in \mbR$, one has
	\bequn	
		\mbP\lp \max_{x \in V_N} h_x^{i} < m_{N_k}+ t \rp \rightarrow \mbE\lb e^{-\frac{e^{-\pi t} }{\pi}Z\lp V_N \rp}\rb.
	\eequn
	And if $M_N- m_N=  \max_{x \in V_N} h_x^{i} - m_N$ converges in law, the random measure $Z\lp dx \rp$ is unique.
\end{proposition}
\end{comment}

\section*{Acknowledgement}
We thank Jian Ding for suggesting this problem, and Marek Biskup, Xin Sun, Yifan Gao, Runsheng Liu, and Haoran Yang for valuable discussions.
The authors are supported by National Key R\&D Program of China (No.\ 2021YFA1002700 and No.\ 2020YFA0712900) and National Natural Science Foundation of China (T2225001).

\bibliographystyle{abbrv}
\bibliography{main}

\begin{thebibliography}{10}

\bibitem{arguin2011genealogy}
L.-P. Arguin, A.~Bovier, and N.~Kistler.
\newblock Genealogy of extremal particles of branching {B}rownian motion.
\newblock {\em Communications on Pure and Applied Mathematics},
  64(12):1647--1676, 2011.

\bibitem{arguin2013extremal}
L.-P. Arguin, A.~Bovier, and N.~Kistler.
\newblock The extremal process of branching {B}rownian motion.
\newblock {\em Probability Theory and related fields}, 157(3-4):535--574, 2013.

\bibitem{biskup2017extrema}
M.~Biskup.
\newblock Extrema of the two-dimensional discrete {G}aussian free field.
\newblock In {\em PIMS-CRM Summer School in Probability}, pages 163--407.
  Springer, 2017.

\bibitem{biskup2016extreme}
M.~Biskup and O.~Louidor.
\newblock Extreme local extrema of two-dimensional discrete {G}aussian free
  field.
\newblock {\em Communications in Mathematical Physics}, 345(1):271--304, 2016.

\bibitem{biskup2018full}
M.~Biskup and O.~Louidor.
\newblock Full extremal process, cluster law and freezing for the
  two-dimensional discrete {G}aussian free field.
\newblock {\em Advances in Mathematics}, 330:589--687, 2018.

\bibitem{biskup2019intermediate}
M.~Biskup and O.~Louidor.
\newblock On intermediate level sets of two-dimensional discrete {G}aussian
  free field.
\newblock In {\em Annales de l'Institut Henri Poincar{\'e}, Probabilit{\'e}s et
  Statistiques}, volume~55, pages 1948--1987. Institut Henri Poincar{\'e},
  2019.

\bibitem{biskup2014conformal}
M.~Biskup and O.~Louidor.
\newblock Conformal symmetries in the extremal process of two-dimensional
  discrete {G}aussian free field.
\newblock {\em Communications in Mathematical Physics}, 375(1):175--235, 2020.

\bibitem{bramson2012tightness}
M.~Bramson and O.~Zeitouni.
\newblock Tightness of the recentered maximum of the two-dimensional discrete
  {G}aussian free field.
\newblock {\em Communications on Pure and Applied Mathematics}, 65(1):1--20,
  2012.

\bibitem{bramson1978maximal}
M.~D. Bramson.
\newblock Maximal displacement of branching {B}rownian motion.
\newblock {\em Communications on Pure and Applied Mathematics}, 31(5):531--581,
  1978.

\bibitem{ding2013exponential}
J.~Ding.
\newblock Exponential and double exponential tails for maximum of
  two-dimensional discrete {G}aussian free field.
\newblock {\em Probability Theory and Related Fields}, 157(1-2):285--299, 2013.

\bibitem{ding2017convergence}
J.~Ding, R.~Roy, O.~Zeitouni, et~al.
\newblock Convergence of the centered maximum of log-correlated {G}aussian
  fields.
\newblock {\em The Annals of Probability}, 45(6A):3886--3928, 2017.

\bibitem{ding2014extreme}
J.~Ding and O.~Zeitouni.
\newblock Extreme values for two-dimensional discrete {G}aussian free field.
\newblock {\em The Annals of Probability}, 42(4):1480--1515, 2014.

\bibitem{georgii2011gibbs}
H.-O. Georgii.
\newblock {\em Gibbs measures and phase transitions}, volume~9.
\newblock Walter de Gruyter, 2011.

\bibitem{kahane1985chaos}
J.-P. Kahane.
\newblock Sur le chaos multiplicatif.
\newblock {\em Ann. Sci. Math. Qu{\'e}bec}, 9(2):105--150, 1985.

\bibitem{kurt2009maximum}
N.~Kurt.
\newblock Maximum and entropic repulsion for a {G}aussian membrane model in the
  critical dimension.
\newblock {\em The Annals of Probability}, 37(2):687--725, 2009.

\bibitem{lalley1987conditional}
S.~P. Lalley and T.~Sellke.
\newblock A conditional limit theorem for the frontier of a branching
  {B}rownian motion.
\newblock {\em The Annals of Probability}, 15(3):1052--1061, 1987.

\bibitem{sun2013uniform}
G.~Lawler, X.~Sun, and W.~Wu.
\newblock Four-dimensional loop-erased random walk.
\newblock {\em The Annals of Probability}, 47(6), 2019.

\bibitem{lawler2013intersections}
G.~F. Lawler.
\newblock {\em Intersections of random walks}.
\newblock Springer Science \& Business Media, 2013.

\bibitem{li2022intermediate}
X.~Li and R.~Liu.
\newblock The intermediate level-sets of the four-dimensional membrane model.
\newblock {\em arXiv preprint arXiv:2205.03621}, 2022.

\bibitem{liggett1978random}
T.~M. Liggett.
\newblock Random invariant measures for {M}arkov chains, and independent
  particle systems.
\newblock {\em Zeitschrift f{\"u}r Wahrscheinlichkeitstheorie und Verwandte
  Gebiete}, 45(4):297--313, 1978.

\bibitem{mckean1975application}
H.~P. McKean.
\newblock Application of {B}rownian motion to the equation of
  {K}olmogorov-{P}etrovskii-{P}iskunov.
\newblock {\em Communications on Pure and Applied Mathematics}, 28(3):323--331,
  1975.

\bibitem{Rodriguez2012PhaseTA}
P.-F. Rodriguez and A.-S. Sznitman.
\newblock Phase transition and level-set percolation for the {G}aussian free
  field.
\newblock {\em Communications in Mathematical Physics}, 320:571--601, 2012.

\bibitem{schramm2009contour}
O.~Schramm and S.~Sheffield.
\newblock Contour lines of the two-dimensional discrete {G}aussian free field.
\newblock {\em Acta Mathematica}, 202(1):21--137, 2009.

\bibitem{schweiger2020maximum}
F.~Schweiger.
\newblock The maximum of the four-dimensional membrane model.
\newblock {\em Annals of Probability}, 48(2):714--741, 2020.

\bibitem{sheffield2007gaussian}
S.~Sheffield.
\newblock Gaussian free fields for mathematicians.
\newblock {\em Probability Theory and Related Fields}, 139(3-4):521--541, 2007.

\bibitem{sheffield2012conformal}
S.~Sheffield and W.~Werner.
\newblock Conformal loop ensembles: the {M}arkovian characterization and the
  loop-soup construction.
\newblock {\em Annals of Mathematics}, 176:1827--1917, 2012.

\bibitem{zeitouni2016branching}
O.~Zeitouni.
\newblock Branching random walks and {G}aussian fields.
\newblock {\em Probability and statistical physics in St. Petersburg},
  91:437--471, 2016.

\end{thebibliography}
\end{document}